%% file: arXiv.tex
\documentclass[a4paper]{article}
\usepackage[left=3.5cm,right=3.0cm,top=2.5cm,bottom=2.5cm]{geometry}
\usepackage{xcolor}
\usepackage{graphicx}
\usepackage{cite}
\usepackage{bm}
\usepackage{amsmath,amssymb,amscd,amsxtra,amsfonts}
\usepackage{lipsum}
\usepackage{amsfonts}
\usepackage{graphicx}
\usepackage{multirow}

\usepackage{array}
\usepackage{amssymb}
\usepackage{verbatim}

\usepackage{xcolor}

\usepackage{epstopdf}
\usepackage{amsmath,bm,stmaryrd}

\usepackage{amsthm}

\usepackage{mathtools}
\usepackage{indentfirst}

\newcommand{\DOmega}{\mathrm{\Omega}}

\newcommand\keywordsname{Key words}
\newcommand\AMSname{AMS subject classifications}

\newenvironment{@abssec}[1]
{\if@twocolumn
\section*{#1}%
\else
\vspace{.05in}\footnotesize
\parindent .2in
{\upshape\bfseries #1. }\ignorespaces
\fi}

{\if@twocolumn\else\par\vspace{.1in}\fi}

\newenvironment{keywords}{\begin{@abssec}{\keywordsname}}{\end{@abssec}}

\input{macros}

\newtheorem{theorem}{Theorem}[section]
\newtheorem{lemma}[theorem]{Lemma}%

\newtheorem{example}{Example}%
\newtheorem{remark}{Remark}%

\numberwithin{equation}{section}

\title{Parametric charge-conservative mixed finite element method for 3D incompressible inductionless MHD equations on curved domains}
\author{Xue Jiang \thanks{School of Mathematics, Statistics and Mechanics, Beijing University of Technology, Beijing 100124.}
\and
Lei Li
\thanks{School of Mathematics, Statistics and Mechanics, Beijing University of Technology, Beijing 100124.}
\and
Lingxiao Li \thanks{School of Mathematics and Statistics, Henan University, Kaifeng 475004.}}
\begin{document}
\date{}
\maketitle

\begin{abstract}
  This paper develops a charge-conservative mixed finite element method with optimal convergence rates for the stationary incompressible inductionless MHD equations on three-dimensional curved domains. The discretization employs the isoparametric Taylor-Hood elements with grad-div stabilization for the velocity-pressure pair, and parametric Brezzi-Douglas-Marini elements for the current density. For sufficiently small meshsize, the discrete inf-sup conditions for both the velocity-pressure and current density-electric potential finite element pairs are established on curved meshes. Utilizing the Piola's transformation, the discrete current density is exactly divergence-free. By employing suitable extensions and projections, optimal a priori error estimates are derived in both the energy norm and the $L^2$-norm. Numerical experiments are presented to confirm the theoretical results.
\end{abstract}


\section{Introduction}
\label{intro}
Magnetohydrodynamics (MHD) equations describe the interaction between electrically conducting fluids and magnetic fields by coupling the fluid dynamics equations and Maxwell equations.
It plays a fundamental role in modeling a wide range of phenomena in plasma physics, astrophysics, and engineering applications involving liquid metals and conducting fluids \cite{pa01,jp19,rm90}. In some practical scenarios, the magnetic Reynolds number may be sufficiently small,
such that the magnetic field induced by the fluid motion can be neglected.
This approximation leads to the so-called inductionless MHD model \cite{gll06,lm89,ss13}.
When all physical quantities no longer exhibit significant temporal variations, the time-derivative term can be neglected, yielding the stationary inductionless MHD equations \cite{ch06,gmp91,pjs88,rp11}.

In this paper, we focus on the stationary incompressible inductionless MHD equations.
The unknowns in the model are the velocity $\Bu$, the current density $\BJ$, the pressure $p$, and the electric potential $\phi$, which satisfy
\begin{subequations}\label{eq:model_1st}
\begin{align}
 \rho\Bu\cdot\nabla\Bu - \nu\Delta\Bu + \nabla p - \BJ\times\BB = \Bf\quad  &\textrm{in}\;\; \DOmega,\label{eq:model_1st:u}\\
\sigma^{-1}\BJ + \nabla\phi - \Bu\times\BB = \Bg\quad    &\textrm{in}\;\;\DOmega,\label{eq:model_1st:J}\\
\Div{\Bu} = 0 \quad &\textrm{in}\;\;\DOmega,\label{eq:model_1st:p}\\
\Div{\BJ} = 0 \quad &\textrm{in}\;\;\DOmega,\label{eq:model_1st:phi}
\end{align}
\end{subequations}
with the following homogeneous boundary conditions for simplicity
\begin{subequations}
\label{eq:model_1st:init_bc}
\begin{align}
\Bu = \bm{0} \quad &\textrm{on}\;\; \partial\DOmega,\\
\phi = 0 \quad &\textrm{on}\;\; \partial\DOmega,
\end{align}
\end{subequations}
where $\DOmega\subset\mathbb{R}^3$ is a bounded Lipschitz domain with possible curved boundary, $\rho$ denotes the fluid density,
 $\nu$ the dynamic viscosity, $\sigma$ the electrical conductivity of fluid, $\BB$ the applied magnetic field, $\Bf$ and $\Bg$ are given source terms.

From the numerical perspective, the incompressible inductionless MHD system presents several intrinsic challenges.
The equations are nonlinear and involve strong multiphysics coupling among velocity, pressure, current density, and electric potential,
together with incompressibility and charge conservation constraints \cite{gll06,rp11}.
In particular, the current density is required to be divergence-free in the absence of internal charge sources,
a condition that is closely related to the physical principle of charge conservation and plays a crucial role in the stability and accuracy of numerical schemes \cite{hlmz18,jlmnr17,nmh1}.

Now let us review some relevant work for this topic. Planas et al. \cite{rp11} proposed a stabilized finite element method based on the variational multiscale framework.
By introducing suitable stabilization terms, their formulation allows equal-order interpolation for the velocity, pressure, current density, and electric potential. This approach circumvented the classical inf-sup conditions associated with the saddle-point problems and demonstrated good robustness in numerical experiments.
With the development of structure-preserving numerical methods, Ni et al. \cite{nmh1,nmh2} proposed collocated grid schemes for the numerical simulation of incompressible inductionless MHD flows. Through a carefully designed discrete coupling between Ohm's law and the electric potential equation, their methods ensure a divergence-free current density at the discrete level on both structured and unstructured meshes. Subsequently, Li et al.\cite{lnz19} proposed a charge-conservative finite element method that strictly ensures the divergence of the current density at the discrete level, and proved the existence of continuous solutions through convergence analysis.
John et al. \cite{jlmnr17} provided a comprehensive review of divergence constraints in mixed finite element methods for incompressible flows,
analyzing their impact on stability and error behavior.
Moreover, a series of works have been devoted to mixed finite element discretization,
nonlinear iterative strategies, and a priori error estimates for incompressible inductionless MHD system,
leading to a more systematic theoretical framework \cite{clsw10,ds04}.

Most of the aforementioned analyses are restricted to polygonal or polyhedral computational domains.
However, practical applications may also involve curved geometries, such as toroidal vacuum vessels in Tokamak devices, curved surfaces of aerospace vehicles,
or biological vascular networks \cite{fokd19,mn22,pbhn20}. Approximating the true geometry by piecewise linear boundaries generally introduces additional geometric errors, which may
deteriorate the overall convergence order and numerical accuracy.

To effectively address problems posed on curved domains, isoparametric finite element methods have been widely adopted \cite{ml86}.
The essential idea is to establish high-order polynomial  mappings between reference and computational elements,
such that the geometric approximation matches the finite element spaces in polynomial degree and then high-order accuracy on curved domains could be achieved \cite{MFEMbook,cpg02}.
In \cite{ml86}, Lenoir systematically investigated the influence of geometric approximation on finite element error estimates for elliptic equations,
clearly identifying the crucial role played by the order of the geometric mapping in the error analysis and laying the theoretical foundation for subsequent developments of
isoparametric finite element methods.
Still for the Poisson equation, Bertrand and Starke constructed parametric Raviart-Thomas elements by directly embedding the geometric mapping into the mixed finite element spaces, with which they provided a systematic analysis of stability and optimal convergence on 3D curved domains \cite{fg16}.
Furthermore, in electromagnetic simulations, studies have shown that geometric approximation errors significantly influence the error behavior of Maxwell equations on curved domains \cite{rc23}.
These findings offer valuable theoretical insights for the numerical analysis of the MHD system in our present work.

This paper is concerned with the numerical approximation of the stationary incompressible inductionless MHD equations on 3D curved domains.
While charge-conservative finite element methods for inductionless MHD have been previously developed on polyhedral domains \cite{lnz19},  their extension to curved geometries is not straightforward. In particular, unlike the polyhedral case, the construction of parametric spaces on curved meshes entails more than the standard isoparametric mapping of the physical coordinates. The transformation of basis functions from the reference element must be defined in accordance with the underlying function space, especially when essential physical constraints are to be preserved at the discrete level. The composition of such field-dependent transformations with the nonlinear geometric approximation of the curved boundary inherently couples the two mappings. This coupling significantly complicates the analysis of discrete inf-sup stability and the derivation of optimal approximation properties on curved meshes. To the best of our knowledge, a rigorous error analysis of charge-conservative mixed finite element methods for inductionless MHD equations on three-dimensional curved domains has not been established.

For the numerical approximation considered in this paper, we employ isoparametric Taylor-Hood elements for the velocity and pressure, parametric Brezzi-Douglas-Marini (BDM) elements for the current density, and a parametric discontinuous Galerkin finite element for the electric potential. The Piola's transformation \cite{fg16} ensures that the divergence-free constraint on the current density is preserved exactly on the discrete curved mesh. To enhance mass conservation, we augment the momentum equation with a grad-div stabilization term \cite{jlmnr17,lz17,or04}.
Following the analytical framework of Sch\"otzau \cite{ds04}, we establish the existence, uniqueness, and appropriate stability bounds of weak solutions. To derive optimal error estimates on curved domains, we adopt an approach that deviates from conventional discretizations based on straight-edged elements \cite{rv21}.
Specifically, we construct the curved computational domain $\DOmega_h$ via a continuous piecewise polynomial mapping $\BM_h$.
Under the uniform boundedness of the mappings and their Jacobians, we derive high-order estimates for the geometric consistency errors arising from the domain approximation. These estimates permit rigorous control of geometric errors and consequently yield optimal a priori estimates in the energy norm. Moreover, by employing a Stokes projection and exploiting an additional order of accuracy in the $H^{-1}$-norm, we establish optimal convergence rates in the $L^2$-norm.
The main novelties of our work are threefold:
\begin{itemize}
\item[(i)] The extension of charge-conservative mixed finite element methods to 3D curved domains for inductionless MHD.
\item[(ii)] The uniform discrete inf-sup stability is proven on curved meshes for both the parametric Taylor-Hood element and current density-electric potential finite element pair.
\item[(iii)] A systematic geometric error analysis is conducted that shows optimal convergence rates in both the energy and $L^2$ norms.
\end{itemize}

The rest of this paper is organized as follows.
In Section \ref{sec:model}, we introduce the continuous mixed variational formulation and establish the well-posedness of the problem.
Section \ref{sec:MFEM} presents the charge-conservative parametric mixed finite element method for the inductionless MHD equations,
where the discrete inf-sup condition is verified and stability bounds are derived.
The error analysis is conducted in Section \ref{sec:error}.
We firstly establish key lemmas addressing geometric approximation and projection errors.
Then we combine them with operator estimates to obtain an energy-norm error bound, and extend it via a Stokes projection to prove optimal $L^2$-norm error estimate.
Numerical experiments validating the theoretical results are provided in Section \ref{sec:exp}.
Finally, concluding remarks are given in Section \ref{sec:con}.

\section{Mathematical model}
\label{sec:model}
Throughout this paper, we adopt standard notation. Let $L^2(\DOmega)$ and $H^1(\DOmega)$ be the usual Lebesgue and Sobolev spaces, with $H_0^1(\DOmega)\subset H^1(\DOmega)$ representing the subspace of functions vanishing on $\partial\DOmega$. For vector fields, we use the spaces $\BH(\Div,\DOmega)$ of square-integrable vectors with square-integrable divergence, $\BL^2(\DOmega) = \left(L^2(\DOmega)\right)^3$, and its divergence-free subspace $\BH(\Div0,\DOmega) = \{\Bv\in\BH(\Div,\DOmega): \Div\Bv=0\}. $
For convenience, we introduce the following notation for the relevant function spaces, the velocity space
\begin{align*}
\BV = \BH_0^1(\DOmega),
\end{align*}
the pressure space
\begin{align*}
Q = L_0^2(\DOmega),
\end{align*}
the current density space
\begin{align*}
\BD = \BH(\Div,\DOmega),
\end{align*}
and the electric potential space
\begin{align*}
S = L^2(\DOmega).
\end{align*}
\subsection{Dimensionless MHD equations}
Let $L, B_0, u_0$  represent the characteristic quantities of length, magnetic induction, and fluid velocity respectively. We introduce the dimensionless variables as follows
\begin{align*}
\Bx\leftarrow\Bx/L, \qquad \Bu\leftarrow\Bu/u_0, \qquad p\leftarrow p/(\rho u_0^2), \qquad \phi\leftarrow\phi/(u_0B_0L),\\
\BB\leftarrow\BB/B_0, \qquad \BJ\leftarrow\BJ/(\sigma u_0B_0), \qquad \Bf\leftarrow\Bf t_0/(\rho u_0), \qquad \Bg\leftarrow\Bg/(u_0B_0).
\end{align*}
The dimensionless incompressible inductionless MHD model is expressed as
\begin{subequations}\label{eq:model_2-0st}
\begin{align}
 \Bu\cdot\nabla\Bu - R_e^{-1}\Delta\Bu + \nabla p - \alpha\BJ\times\BB = \Bf\quad  &\textrm{in}\;\; \DOmega,\\
\BJ + \nabla\phi - \Bu\times\BB = \Bg\quad    &\textrm{in}\;\;\DOmega,\\
\Div{\Bu} = 0 \quad &\textrm{in}\;\;\DOmega,\\
\Div{\BJ} = 0 \quad &\textrm{in}\;\;\DOmega,
\end{align}
\end{subequations}
where $R_e = \rho Lu_0/\nu$ is the Reynolds number and $\alpha = \sigma LB_0^2/(\rho u_0)$  is the coupling number between fluid and magnetic field. In the subsequent analysis, we assume that all physical parameters are constants and $\BB\in\BL^\infty(\DOmega)$.

To improve the mass conservation, we also introduce a stabilization parameter $\gamma>0$ and reformulate \eqref{eq:model_2-0st} into an AL form (also see \cite{lnz19,or04})
\begin{subequations}\label{eq:model_2st}
\begin{align}
 \Bu\cdot\nabla\Bu - R_e^{-1}\Delta\Bu - \gamma\nabla\Div\Bu + \nabla p - \alpha\BJ\times\BB = \Bf\quad  &\textrm{in}\;\; \DOmega,\label{eq:model_2st:u}\\
\BJ + \nabla\phi - \Bu\times\BB = \Bg\quad    &\textrm{in}\;\;\DOmega,\label{eq:model_2st:J}\\
\Div{\Bu} = 0 \quad &\textrm{in}\;\;\DOmega,\label{eq:model_2st:p}\\
\Div{\BJ} = 0 \quad &\textrm{in}\;\;\DOmega,\label{eq:model_2st:phi}\\
\Bu = \bm{0} \quad &\textrm{on}\;\; \partial\DOmega,\label{eq:model_2st:u_bc}\\
\phi = 0 \quad &\textrm{on}\;\; \partial\DOmega.\label{eq:model_2st:phi_bc}
\end{align}
\end{subequations}
The other advantage of this approach lies in its ability to improve the convergence rate of the preconditioned iterative method \cite{lz17,lnz19}.
In the remainder of this paper, we will focus on the AL form \eqref{eq:model_2st} rather than the original problem.
\begin{remark}
In the continuous form, $\Div\Bu=0$ removes the grad-div term, but the discrete velocity $\Bu_h$ is usually not divergence-free
when using Taylor-Hood elements.
Retaining stabilization term in the discrete problem will improve stability and mass conservation while maintaining consistency \cite{jlmnr17,or04}.
\end{remark}

\subsection{Solenoidal function spaces}
The divergence-free subspaces of $\BV$ and $\BD$ are denoted as
\begin{align*}
\BV(\Div 0) = \BV\cap\BH(\Div0,\DOmega),\;\BD(\Div 0) = \BD\cap\BH(\Div0,\DOmega).
\end{align*}
In terms of these solenoidal spaces, we then introduce an auxiliary variational problem for \eqref{eq:model_2st}: find $(\Bu,\BJ)\in\BV(\Div 0)\times\BD(\Div 0)$
\begin{subequations}\label{eq:model_3st}
\begin{align}
\Co(\Bu;\Bu,\Bv) + R_e^{-1}(\nabla\Bu,\nabla\Bv) - \alpha(\BJ\times\BB,\Bv) = (\Bf,\Bv),\;\;\forall\Bv\in\BV(\Div 0),\label{eq:model_3st:u}\\
(\BJ,\Bd) - (\Bu\times\BB,\Bd) = (\Bg,\Bd),\;\;\forall\Bd\in\BD(\Div 0),\label{eq:model_3st:J}
\end{align}
\end{subequations}
where the trilinear form is defined as
\begin{align}\label{eq:tri_define}
\Co(\Bw;\Bu,\Bv) := \frac{1}{2}(\Bw\cdot\nabla\Bu,\Bv) - \frac{1}{2}(\Bw\cdot\nabla\Bv,\Bu).
\end{align}
\begin{lemma}
The following results hold:\\
(1) Let $\Bw, \Bu, \Bv$ in $\BV$, we have
\begin{align}\label{eq:lemma_O_con}
&|\Co(\Bw;\Bu,\Bv)|\nn\\
\leq & \frac{1}{2}\big(\|\Bw\|_{\BL^4(\DOmega)}\|\nabla\Bu\|_{\BL^2(\DOmega)}\|\Bv\|_{\BL^4(\DOmega)}
+ \|\Bw\|_{\BL^4(\DOmega)}\|\nabla\Bv\|_{\BL^2(\DOmega)}\|\Bu\|_{\BL^4(\DOmega)}\big)\nn\\
\leq & C_{\Co}\|\Bw\|_{\BH^1(\DOmega)}\|\Bu\|_{\BH^1(\DOmega)}\|\Bv\|_{\BH^1(\DOmega)}.
\end{align}
(2) Let $\Bw, \Bv, \Bv$ in $\BV$, we have
\begin{align}\label{eq:lemma_O_pro}
\Co(\Bw;\Bv,\Bv) = \frac{1}{2}(\Bw\cdot\nabla\Bv,\Bv) - \frac{1}{2}(\Bw\cdot\nabla\Bv,\Bv) = 0.
\end{align}
\end{lemma}

For further analysis,  we add \eqref{eq:model_3st:u} and $\alpha\times$\eqref{eq:model_3st:J}, which yields
\begin{equation}\label{eq:model_4st}
\Ca(\Bu,\BJ;\Bv,\Bd) + \Co(\Bu;\Bu,\Bv) = \Cl(\Bf,\Bg;\Bv,\Bd)
\end{equation}
for all $(\Bv,\Bd)\in\BV(\Div 0)\times\BD(\Div 0)$, where
\begin{align}
\Ca(\Bu,\BJ;\Bv,\Bd) &= R_e^{-1}(\nabla\Bu,\nabla\Bv) - \alpha(\BJ\times\BB,\Bv) + \alpha(\BJ,\Bd) - \alpha(\Bu\times\BB,\Bd),\label{eq:define_A}\\
\Cl(\Bf,\Bg;\Bv,\Bd) &= (\Bf,\Bv) + \alpha(\Bg,\Bd)\label{eq:define_B}.
\end{align}
Next, we define
\begin{align*}
&|||(\Bv,\Bd)|||_{\Ca}^2 = \|\Bv\|_{\BH^1(\DOmega)}^2 + \|\Bd\|_{\BH(\Div,\DOmega)}^2,\\
&|||\Cl|||_{a} = \sup\limits_{(\bm{0},\bm{0})\neq(\Bv,\Bd)\in\BV(\Div 0)\times\BD(\Div 0)}\frac{\Cl(\Bf,\Bg;\Bv,\Bd)}{|||(\Bv,\Bd)|||_{\Ca}},\\
&|||(\Bf,\Bg)|||_{\Cf}^2 = \|\Bf\|_{\BL^2(\DOmega)}^2 + \|\Bg\|_{\BL^2(\DOmega)}^2.
\end{align*}
It follows that $|||\Cl|||_{a}\leq\max\{1,\alpha\}|||(\Bf,\Bg)|||_{\Cf}$.

Applying the Cauchy-Schwarz inequality, the forms $\Ca, \Co, \Cl$ have the following properties:
\begin{align}
\label{eq:lemma_coer_con}
\Ca(\Bv,\Bd;\Bv,\Bd)
&= R_e^{-1}(\nabla\Bv,\nabla\Bv) - \alpha(\Bd\times\BB,\Bv) + \alpha(\Bd,\Bd) - \alpha(\Bv\times\BB,\Bd)\nonumber\\
&= R_e^{-1}\|\nabla\Bv\|_{\BL^2(\DOmega)}^2 + \alpha\|\Bd\|_{\BL^2(\DOmega)}^2\nonumber\\
&\geq C_a\min\{R_e^{-1},\alpha\}|||(\Bv,\Bd)|||_{\Ca}^2,\\
\label{eq:lemma_con_con}
|\Ca(\Bu,\BJ;\Bv,\Bd)|
&= |R_e^{-1}(\nabla\Bu,\nabla\Bv) - \alpha(\BJ\times\BB,\Bv) + \alpha(\BJ,\Bd) - \alpha(\Bu\times\BB,\Bd)|\nonumber\\
&\leq R_e^{-1}\|\nabla\Bu\|_{\BL^2(\DOmega)}\|\nabla\Bv\|_{\BL^2(\DOmega)} + \alpha\|\BJ\|_{\BL^2(\DOmega)}\|\BB\|_{\BL^\infty(\DOmega)}\|\Bv\|_{\BL^2(\DOmega)} \nonumber\\
&+ \alpha\|\BJ\|_{\BL^2(\DOmega)}\|\Bd\|_{\BL^2(\DOmega)} + \alpha\|\Bu\|_{\BL^2(\DOmega)}\|\BB\|_{\BL^\infty(\DOmega)}\|\Bd\|_{\BL^2(\DOmega)}\nn\\
&\leq R_e^{-1}\|\Bu\|_{\BH^1(\DOmega)}\|\Bv\|_{\BH^1(\DOmega)} + \alpha {\|\BB\|_{\BL^\infty(\DOmega)}}\|\BJ\|_{\BH(\Div,\DOmega)}\|\Bv\|_{\BH^1(\DOmega)} \nn\\
&+ \alpha\|\BJ\|_{\BH(\Div,\DOmega)}\|\Bd\|_{\BH(\Div,\DOmega)} + \alpha {\|\BB\|_{\BL^\infty(\DOmega)}}\|\Bu\|_{\BH^1(\DOmega)}\|\Bd\|_{\BH(\Div,\DOmega)}\nn\\
&\leq \max\{R_e^{-1},\alpha, \alpha {\|\BB\|_{\BL^\infty(\DOmega)}}\}|||(\Bu,\BJ)|||_{\Ca}|||(\Bv,\Bd)|||_{\Ca},\\
\label{eq:lemma_O_conn}
|\Co(\Bw;\Bu,\Bv)|
&\leq C_{\Co}\|\Bw\|_{\BH^1(\DOmega)}\|\Bu\|_{\BH^1(\DOmega)}\|\Bv\|_{\BH^1(\DOmega)}\nn\\
&\leq C_{\Co}|||(\Bw,\Bd)|||_{\Ca}|||(\Bu,\BJ)|||_{\Ca}|||(\Bv,\Bd)|||_{\Ca},\\
\label{eq:lemma_bound_con}
|\Cl(\Bf,\Bg;\Bv,\Bd)|
&\leq |||\Cl|||_{a} |||(\Bv,\Bd)|||_{\Ca}, \\
\label{eq:lemma_boundd_con}
|\Cl(\Bf,\Bg;\Bv,\Bd)|
&\leq \|\Bf\|_{\BL^2(\DOmega)}\|\Bv\|_{\BH^1(\DOmega)} + \alpha\|\Bg\|_{\BL^2(\DOmega)}\|\Bd\|_{\BH(\Div,\DOmega)}\nn\\
&\leq \max\{1,\alpha\}|||(\Bf,\Bg)|||_{\Cf} |||(\Bv,\Bd)|||_{\Ca},
\end{align}
where $C_a, C_{\Co}$ are positive constants depending on $\DOmega$.

We will discuss the existence and uniqueness of solutions in solenoidal function spaces based on the results from \cite{NSbook} (Section \uppercase\expandafter{\romannumeral 4}, Theorem 1.2, 1.3) and \cite{ds04} (Theorem 2.11). For clarity, we show the following theorem.
\begin{theorem}\label{exit_unique}
Let $V$ be a separable Hilbert space with dual space $V'$. Assume that $l\in V'$ is a given linear continuous functional and
$$a(\cdot;\cdot,\cdot):V \times V \times V \rightarrow \mathbb{R}$$
is a trilinear form satisfying the following properties:
\begin{itemize}
\item[(1)] There exists a constant $ \epsilon > 0$ such that
      \begin{align*}
      |a(u; v, w)| \leq \epsilon \|u\|_V \|v\|_V \|w\|_V, \quad \forall u, v, w \in V.
      \end{align*}
\item[(2)] There exists a constant $ \beta > 0 $ such that
      \begin{align*}
      a(u; v, v) \geq \beta \|v\|_V^2, \quad \forall u, v \in V.
      \end{align*}
\item[(3)] The mapping $u \mapsto a(u; u, v) $ is sequentially weakly continuous on $V$.
  Namely, for any sequence $\{u_m\}_{m \in \mathbb{N}} \subset V$ satisfying $u_m \rightarrow u$ in $V$ as $m \rightarrow \infty$,
      \begin{align*}
      \lim_{m \to \infty} a(u_m; u_m, v) = a(u; u, v), \forall v\in V.
      \end{align*}
\end{itemize}
Consider the variational problem: find $u\in V$ such that
\begin{align*}
a(u; u, v) = l(v), \quad\forall v\in V.
\end{align*}
Then, under the small parameter condition
\begin{align*}
\epsilon \beta^{-2} \|l\|_{V'} < 1
\end{align*}
the variational problem has a unique solution $u\in V$. Moreover, any solution satisfies the stability bound
\begin{align*}
\|u\|_V \leq \beta^{-1} \|l\|_{V'}.
\end{align*}
\end{theorem}

The corresponding result for variational problem \eqref{eq:model_4st} is stated as follows.
\begin{theorem}\label{thm_div0_exist_con}
For given $\Bf, \Bg\in\BL^2(\DOmega)$, there exists at least one solution
of variational problem \eqref{eq:model_4st}. Moreover, such solution satisfies the stability bound
\begin{align*}
|||(\Bu,\BJ)|||_{\Ca}\leq\frac{|||\Cl|||_{a}}{C_a\min\{R_e^{-1},\alpha\}}.
\end{align*}
\end{theorem}
\begin{proof}
The proof consists of the following steps.

{\bf Step 1.} The spaces $\BV(\Div 0)$ and $\BD(\Div 0)$ are closed subspaces of $\BH^1(\DOmega)$ and $\BH(\Div,\DOmega)$, respectively. Consequently, $\BV(\Div 0)\times\BD(\Div 0)$ is separable.

{{\bf Step 2.} Let $\{(\Bu_m, \BJ_m)\}_{m\in\bbN}$ be a sequence in $\BV(\Div\,0)\times \BD(\Div\,0)$ converging weakly to $(\Bu,\BJ)$. Weak convergence implies boundedness in $\BV(\Div\,0)\times\BD(\Div\,0)$. The weak continuity of the bilinear form $\Ca$ follows immediately from its continuity.

It remains to verify the weak continuity of the nonlinear convection term. By the Rellich-Kondrachov compactness theorem, the embedding $\BH^1(\DOmega)^3 \hookrightarrow \BL^4(\DOmega)^3$ is compact. Hence, upon extracting a subsequence, still denoted by $\Bu_m$, we have $\Bu_m \to \Bu$ strongly in $\BL^4(\DOmega)^3$, and the convergence of $\BJ_m$ to $\BJ$ is inherited by the subsequence. In view of the boundedness of $\{\Bu_m\}$ in $\BH^1(\DOmega)^3$ and the continuity of $\Co$ established in \eqref{eq:lemma_O_conn}, it follows that
\[
\lim_{m\to\infty} \Co(\Bu_m;\Bu_m,\Bv) = \Co(\Bu;\Bu,\Bv),
\quad \forall \Bv\in \BV(\Div\,0).
\]
Since the limit is independent of the chosen subsequence, the subsequence convergence criterion ensures convergence of the whole sequence. Consequently, the mapping
\[
(\Bu,\BJ) \mapsto \Ca(\Bu,\BJ;\Bv,\Bd) + \Co(\Bu;\Bu,\Bv)
\]
is sequentially weakly continuous on $\BV(\Div\,0)\times \BD(\Div\,0)$.

}

{\bf Step 3.} The coercivity and continuity properties are shown in \eqref{eq:lemma_coer_con}-\eqref{eq:lemma_bound_con}. Combining these properties with the results of Steps 1 and 2 completes the proof.
\end{proof}

Given that the coefficients satisfy certain conditions, the problem admits a unique solution according to Theorem \ref{exit_unique}.
\begin{theorem}\label{thm_div0_unique_con}
Assume that
\begin{align}\label{div0_unique_con}
\frac{C_{\Co}|||\Cl|||_{a}}{C_a^2\min\{R_e^{-2},\alpha^2\}} \leq 1,
\end{align}
the variational problem \eqref{eq:model_4st} has a unique solution.
\end{theorem}

\subsection{Continuous mixed variational formulation}
Next, we define the continuous mixed variational formulation for \eqref{eq:model_2st}: find $(\Bu,p,\BJ,\phi)\in\BV\times Q\times\BD\times S$ such that
\begin{subequations}\label{eq:model_5st}
\begin{align}
\Co(\Bu;\Bu,\Bv) + R_e^{-1}(\nabla\Bu,\nabla\Bv) + \gamma(\Div\Bu,\Div\Bv)\nn\\
 - (p,\Div\Bv)- \alpha(\BJ\times\BB,\Bv)
= (\Bf,\Bv),\label{eq:model_5st:u}\\
(\BJ,\Bd) - (\phi,\Div\Bd) - (\Bu\times\BB,\Bd) = (\Bg,\Bd),\label{eq:model_5st:J}\\
-(\Div{\Bu},q) = 0,\label{eq:model_5st:p}\\
-(\Div{\BJ},\varphi) = 0.\label{eq:model_5st:phi}
\end{align}
\end{subequations}
for all $(\Bv,q,\Bd,\varphi)\in\BV\times Q\times\BD\times S$.
We introduce the bilinear forms
\begin{align*}
b_q(q,\Bv) := (q,\Div\Bv),\quad b_\varphi(\varphi,\Bd) := (\varphi,\Div\Bd)
\end{align*}
which were neglected in previous analyses. $b_q(\cdot,\cdot)$ and $b_\varphi(\cdot,\cdot)$ are both continuous and satisfy the following inf-sup conditions
\begin{align}
\label{infsup_con_up}
\sup_{\bm{0}\neq\Bv\in\BV}\frac{b_q(q,\Bv)}{\|\Bv\|_{\BH^1(\DOmega)}}\geq\beta_1\|q\|_{L^2(\DOmega)}>0,\\
\label{infsup_con_Jphi}
\sup_{\bm{0}\neq\Bd\in\BD}\frac{b_\varphi(\varphi,\Bd)}{\|\Bd\|_{\BH(\Div,\DOmega)}}\geq\beta_2\|\varphi\|_{L^2(\DOmega)}>0,
\end{align}
respectively. The detailed proof can be found in \cite{MFEMbook} (Section 4.4.2) and \cite{NSbook} (Section \uppercase\expandafter{\romannumeral 1}.5.1).

Define $\Cb(q,\varphi;\Bv,\Bd) := b_q(q,\Bv) + \alpha b_\varphi(\varphi,\Bd)$, we have
\begin{subequations}\label{eq:model_6st}
\begin{align}
\hspace{-0.4cm}\Ca(\Bu,\BJ;\Bv,\Bd) + \Co(\Bu;\Bu,\Bv) - \Cb(p,\phi;\Bv,\Bd) &= \Cl(\Bf,\Bg,\Bv,\Bd),\label{eq:model_6st:1}\\
\Cb(q,\varphi;\Bu,\BJ) &= 0,\label{eq:model_6st:2}
\end{align}
for all $(\Bv,q,\Bd,\varphi)\in\BV\times Q\times\BD\times S$.
\end{subequations}
By setting
\begin{align*}
|||(q,\varphi)|||_{\Cf}^2 = \|q\|_{L^2(\DOmega)}^2 + \|\varphi\|_{L^2(\DOmega)}^2,
\end{align*}
it is easy to obtain
\begin{align*}
|\Cb(q,\varphi;\Bv,\Bd)|\leq C_{\Cf}|||(q,\varphi)|||_{\Cf}|||(\Bv,\Bd)|||_{\Ca}.
\end{align*}
\begin{lemma}\label{lemma_B_infsup_con}
There exists a constant $\Gamma>0$ only depending on $\DOmega$ such that for all $(q,\varphi)\in Q\times S$,
\begin{align*}
\sup_{(\Bv,\Bd)\in\BV\times\BD}\frac{\Cb(q,\varphi;\Bv,\Bd)}{|||(\Bv,\Bd)|||_{\Ca}}
\geq \Gamma|||(q,\varphi)|||_{\Cf}.
\end{align*}
\end{lemma}
\begin{proof}
From the inf-sup condition \eqref{infsup_con_up} and \eqref{infsup_con_Jphi}, we can select $\Bv\in\BV, \Bd\in\BD$ satisfying
\begin{align*}
\|\Bv\|_{\BH^1(\DOmega)} \leq \beta_1^{-1}\|q\|_{L^2(\DOmega)},\quad
\|\Bd\|_{\BH(\Div,\DOmega)} \leq \beta_2^{-1}\|\varphi\|_{L^2(\DOmega)},
\end{align*}
and
\begin{align*}
b_q(q,\Bv) \geq \|q\|_{L^2(\DOmega)}^2,\quad b_\varphi(\varphi,\Bd)\geq\|\varphi\|_{L^2(\DOmega)}^2.
\end{align*}
Using the definition of $\Cb$ and the above inequalities,
\begin{align*}
\Cb(q,\varphi;\Bv,\Bd) &\geq \|q\|_{L^2(\DOmega)}^2 + \alpha\|\varphi\|_{L^2(\DOmega)}^2
\geq\min\{1,\alpha\} |||(q,\varphi)|||_{\Cf}^2,
\end{align*}
and
\begin{align*}
|||(\Bv,\Bd)|||_{\Ca}^2 &\leq \max\{\beta_1^{-2},\beta_2^{-2}\}|||(q,\varphi)|||_{\Cf}^2.
\end{align*}
Thus, we conclude
\begin{align*}
\frac{\Cb(q,\varphi;\Bv,\Bd)}{|||(\Bv,\Bd)|||_{\Ca}}
\geq \frac{\min\{1,\alpha\}}{\sqrt{\max\{\beta_1^{-2},\beta_2^{-2}\}}}|||(q,\varphi)|||_{\Cf}.
\end{align*}
Taking $\Gamma = \frac{\min\{1,\alpha\}}{\sqrt{\max\{\beta_1^{-2},\beta_2^{-2}\}}}>0$ completes the proof.
\end{proof}
\begin{theorem}
For $\Bf, \Bg\in\BL^2(\DOmega)$, there exists at least one solution of the mixed variational problem \eqref{eq:model_5st}.
The following stability bounds hold
\begin{align*}
|||(\Bu,\BJ)|||_{\Ca}\leq \frac{|||\Cl|||_{a}}{C_a\min\{R_e^{-1},\alpha\}}
\end{align*}
and
\begin{align*}
|||(p,\phi)|||_{\Cf}\leq &\Gamma^{-1}\big[\max\{R_e^{-1},\alpha, \alpha {\|\BB\|_{\BL^\infty(\DOmega)}},\gamma\}|||(\Bu,\BJ)|||_{\Ca}\\
&+ C_{\Co}|||(\Bu,\BJ)|||_{\Ca}^2 + \max\{1,\alpha\}|||(\Bf,\Bg)|||_{\Cf}\big]
\end{align*}
for any solution $(\Bu,p,\BJ,\phi)$. Further, under the assumption \eqref{div0_unique_con}, the solution of \eqref{eq:model_5st} is unique.
\end{theorem}
\begin{proof}
For $(\Bu,\BJ)\in\BV(\Div 0)\times\BD(\Div 0)$, \eqref{eq:model_6st} can be rewritten as: find $(p,\phi)\in Q\times S$ satisfy
\begin{align}
\label{eq:model_B}
\Cb(p,\phi;\Bv,\Bd) = \Ca(\Bu,\BJ;\Bv,\Bd) + \Co(\Bu;\Bu,\Bv) - \Cl(\Bf,\Bg;\Bv,\Bd)
\end{align}
for all $(\Bv,\Bd)\in(\BV\times\BD)/(\BV(\Div 0)\times\BD(\Div 0))$.

According to the inf-sup condition in Lemma \ref{lemma_B_infsup_con}, the problem \eqref{eq:model_B} has a unique solution. Using the continuity of $\Cl, \Ca, \Co$, we have
\begin{align*}
&\Gamma|||(p,\phi)|||_{\Cf}\\
\leq&\sup_{(\Bv,\Bd)\in\BV\times\BD}\frac{\Cb(p,\phi;\Bv,\Bd)}{|||(\Bv,\Bd)|||_{\Ca}}\\
\leq&\sup_{(\Bv,\Bd)\in\BV\times\BD}\frac{\Ca(\Bu,\BJ;\Bv,\Bd) + \Co(\Bu;\Bu,\Bv) - \Cl(\Bf,\Bg;\Bv,\Bd)}{|||(\Bv,\Bd)|||_{\Ca}}\\
\leq&\max\{R_e^{-1},\alpha, \alpha {\|\BB\|_{\BL^\infty(\DOmega)}},\gamma\}|||(\Bu,\BJ)|||_{\Ca}+C_{\Co}|||(\Bu,\BJ)|||_{\Ca}^2\\
& + \max\{1,\alpha\}|||(\Bf,\Bg)|||_{\Cf}.
\end{align*}
The proof is completed.
\end{proof}

\section{Charge-conservative parametric mixed finite element method}
\label{sec:MFEM}
\subsection{Meshes and mapping}
For clarity, we denote the original physical domain by $\DOmega$, the polyhedral reference domain by $\wh\DOmega_h$ and the approximated domain by $\DOmega_h$.
Here, $\wh\DOmega_h$ serves as an intermediate bridge, establishing a connection between $\DOmega$ and $\DOmega_h$.
Indeed, let $\wh\Ct_h$ be the the standard straight-edge triangulation of $\DOmega$.
Then $\wh\DOmega_h$ is the domain occupied by $\wh\Ct_h$.
We define $h = \max_{\widehat{K}\in \wh\Ct_h}\mathrm{diam}(\widehat{K}) $.
While $\DOmega_h$ is obtained from $\wh\DOmega_h$ through a continuous piecewise polynomial mapping $\BM_h$ of degree $k$.
In the lowest-order case ($k=1$), we have $\DOmega_h=\wh\DOmega_h$.
The triangulation of $\DOmega_h$ is denoted by $\Ct_h$, where each element is the image of a straight-sided reference tetrahedron under the mapping $\BM_h$. Further, triangulation of $\DOmega$ is denoted by $\widetilde{\Ct}_h$.

{For the error analysis, we make the following standard assumptions on the
geometry and the mesh \cite{fg16,ml86}:

\begin{enumerate}
\item[(A1)]  The mesh $\widehat{\Ct}_h$ is conforming, shape-regular and quasi-uniform tetrahedral mesh:
there exists a constant $\sigma>0$, independent of $h$, such that
$h_{\widehat{K}}/\rho_{\widehat{K}}\le \sigma$ and $h/\rho_{\widehat{K}}\le \sigma$ for all elements
$\widehat{K} \in \widehat{\Ct}_h$, where $h_{\widehat{K}}$ is the diameter of $\widehat{K}$ and
$\rho_{\widehat{K}}$ the radius of its inscribed ball.
\item[(A2)]  Each boundary face of $\widetilde{\Ct}_h$ is contained in one of the $C^{k+1}$ patches of $\partial\DOmega$, and each element of $\widetilde{\Ct}_h$ has at most one boundary face.
\item[(A3)]  There exists a continuous and invertible mapping $\BM:\wh{\DOmega}_h\to\DOmega$ with $\BM\in C^{k+1}(\wh{\DOmega}_h)^3$.
\item[(A4)] For sufficiently small $h$, the isoparametric mapping $\BM_h$ is bijective and satisfies the uniform bounds \eqref{eq:property_1st:Mh}, with constants independent of $h$.
\item[(A5)] The Jacobian determinant of $\BM_h$ satisfies $\det\bbJ_{\BM_h}\ge\delta$ for some positive constant $\delta$ independent of $h$.
\end{enumerate}
}

By the assumptions and standard error estimates \cite{TEbook}, we have
\begin{align}\label{eq:property_2st}
\|\BM_h-\BM\|_{\BL^{\infty}(\wh\DOmega_h)}\lesssim h^{k+1},\\
\label{eq:property_1st:M}
\|\bbJ_{\BM}\|_{\BW^{k,\infty}(\wh\DOmega_h)}\lesssim1,\;\;\;
\|\bbJ_{\BM^{-1}}\|_{\BW^{k,\infty}(\DOmega)}\lesssim1,
\end{align}
where $\bbJ_{\BM}$ and $\bbJ_{\BM^{-1}}$ denote the Jacobian of $\BM$ and $\BM^{-1}$, respectively. For $k>1$, $\BM_h$ differs from the identity mapping.
Under the smoothness assumption of $\partial\DOmega$, the distance between the approximate boundary $\partial\DOmega_h$ and the true boundary $\partial\DOmega$  is of order $k+1$. {For sufficiently small $h$, the mapping $\BM_h$ is invertible and its Jacobian $\bbJ_{\BM_h}$ satisfies the estimates \cite{fsg14}
\begin{align}\label{eq:property_1st:Mh}
\|\bbJ_{\BM_h}\|_{\BW^{k,\infty}(\wh\DOmega_h)}\lesssim1,\;\;\;
\|\bbJ_{\BM_h^{-1}}\|_{\BW^{k,\infty}(\DOmega_h)}\lesssim1,\;\;\;
\|\bbJ_{\BM_h}-\bbI\|_{\BL^{\infty}(\wh\DOmega_h)}\lesssim h.
\end{align}}

Due to the loss of one order of accuracy when approximating derivatives, it follows that
\begin{align}
\label{eq:property_3st}
\|\bbJ_{\BM_h}-\bbJ_{\BM}\|_{\BL^{\infty}(\wh\DOmega_h)}\lesssim h^k.
\end{align}
\begin{remark}
Accordingly, it is convenient to introduce the compound mapping $\Phi_h=\BM_h\circ\BM^{-1}$, which maps $\DOmega$ to $\DOmega_h$.
\end{remark}

\begin{remark}
Throughout the paper, the notation $\lesssim $ means that there exists a constant $C>0$, independent of $h$ and $k$, such that $a\le C b$. The constant $C$ may depend on $\DOmega$, and the physical parameters $R_e$, $\alpha$, $\gamma$, and $\|\BB\|_{\BL^\infty}$; such dependence is indicated explicitly where relevant.
\end{remark}

\subsection{Preliminaries}
The standard finite element spaces for the discrete velocity, pressure, current density, and electric potential on the straight-edged tetrahedral mesh $\widehat{\Ct}_h$ are defined as
($k \geq 2$)
\begin{align*}
&\wh\BV_h^k = \{\hBv_h\in\BH^1(\wh\DOmega_h):\hBv_h|_{\hat{K}}\in\BP_k(\hat{K}),\ \forall \hat{K}\in\wh\Ct_h\},\\
&\wh Q_h^{k-1} = \{\hat{q}_h\in H^1(\wh\DOmega_h):\hat{q}_h|_{\hat{K}}\in P_{k-1}(\hat{K}),\ \forall \hat{K}\in\wh\Ct_h\},\\
&\wh\BD_h^{k-1} = \{\hBd_h\in\BH(\wh\Div,\wh\DOmega_h):\hBd_h|_{\hat{K}}\in\BP_{k-1}(\hat{K}),\ \forall \hat{K}\in\wh\Ct_h\},\\
&\wh S_h^{k-2} = \{\hat{\varphi}_h\in L^2(\wh\DOmega_h):\hat{\varphi}_h|_{\hat{K}}\in P_{k-2}(\hat{K}),\ \forall \hat{K}\in\wh\Ct_h\}.
\end{align*}
These spaces satisfy the following inf-sup conditions \cite{MFEMbook,NSbook}
\begin{align}
\label{infsup_dis_up_hat}
\sup_{\bm{0}\neq\hBv_h\in\wh\BV_{h,0}^k}\frac{b_q(\hat{q}_h,\hBv_h)_{\hat{h}}}{\|\hBv_h\|_{\BH^1(\wh\DOmega_h)}}
\geq\beta_3\|\hat{q}_h\|_{L^2(\wh\DOmega_h)}>0,\\
\label{infsup_dis_Jphi_hat}
\sup_{\bm{0}\neq\hBd_h\in\wh\BD_h^{k-1}}\frac{b_\varphi(\hat{\varphi}_h,\hBd_h)_{\hat{h}}}{\|\hBd_h\|_{\BH(\wh\Div,\wh\DOmega_h)}}
\geq\beta_4\|\hat{\varphi}_h\|_{L^2(\wh\DOmega_h)}>0,
\end{align}
where $\beta_3, \beta_4$ are positive constants depending only on $\wh\DOmega_h$.

Based on the Piola's transformation \cite{fg16} and isoparametric mapping $\BM_h$,
the parametric finite element spaces for $(\Bu_h, p_h, \BJ_h, \phi_h)$ on the curved mesh $\Ct_h$ are defined as
\begin{align*}
\BV_h^k &= \{\hBv_h\circ\BM_h^{-1}:\hBv_h\in\wh\BV_h^k\},\\
Q_h^{k-1} &= \{\hat{q}_h\circ\BM_h^{-1}:\hat{q}_h\in\wh Q_h^{k-1}\},\\
\BD_h^{k-1} &= \{\frac{1}{\det{\bbJ}_{\BM_h}}{\bbJ}_{\BM_h}\hBd_h\circ\BM_h^{-1}:\hBd_h\in\wh\BD_h^{k-1} \},\\
S_h^{k-2} &= \{\hat{\varphi}_h\circ\BM_h^{-1}:\hat{\varphi}_h\in\wh S_h^{k-2}\}.
\end{align*}
For notational convenience, we introduce the subspaces
\begin{align*}
\wh\BV_{h,0}^k &= \{\hBv_h\in\wh\BV_h^k:\hBv_h=0\; \textrm{on}\; \partial\wh\DOmega_h\},\\
\BV_{h,0}^k &= \{\Bv_h\in\BV_h^k:\Bv_h=0\; \textrm{on}\; \partial\DOmega_h\},\\
\BV_{h,0}^k(\Div 0)&=\{\Bv_h\in\BV_{h,0}^k:b_q(q_h,\Bv_h)_h = 0,\forall q_h\in Q_h^{k-1}\},\\
\BD_h^{k-1}(\Div 0)&=\{\Bd_h\in\BD_h^{k-1}:b_\varphi(\varphi_h,\Bd_h)_h = 0,\forall \varphi_h\in S_h^{k-2}\}.
\end{align*}
It should be noted that the functions in the spaces $\BV_{h,0}^k, Q_h^{k-1}, \BD_h^{k-1}$, and $S_h^{k-2}$ are no longer piecewise polynomials, but rather compositions of polynomial functions and the inverse of the mapping $\BM_h$.

\begin{remark}
The discrete pressure space $Q_h^{k-1}$ does not explicitly enforce the zero-mean condition; the pressure is determined up to a constant and fixed by post-processing. The potential space is $S_h^{k-2}\subset L^2(\DOmega_h)$, and its homogeneous boundary condition is enforced weakly through the variational formulation, consistent with the continuous setting.
\end{remark}

According to the chain rule,
\begin{align}
&\nabla\Bv_h(\Bx) =\wh\nabla\hBv_h(\hBx)\bbJ_{\BM_h}^{-1},\label{eq:trans_grad_v}\\
&\Div\Bv_h(\Bx) = \frac{1}{\det\bbJ_{\BM_h}}\wh\Div((\det\bbJ_{\BM_h})\hBv_h\bbJ_{\BM_h}^{-1})(\hBx),\ \forall\Bv_h\in\BV_{h,0}^k,\label{eq:trans_div_v}\\
&\Div\Bd_h(\Bx) = \frac{1}{\det\bbJ_{\BM_h}}\wh\Div\,\hBd_h(\hBx),\ \forall\Bd_h\in\BD_h^{k-1},\label{eq:trans_div_d}
\end{align}
where $\wh\nabla$ and $\wh\Div$ are the gradient and divergence operators with respect to $\hBx$, and all vector functions are understood as row vectors.

\begin{remark}
To avoid confusion, unlike $(\cdot,\cdot)$, which denotes the inner product on $\DOmega$, the symbol $(\cdot,\cdot)_h$ represents the inner products on $\DOmega_h$.
\end{remark}
\begin{remark}\label{rem:divuh}
The divergence of velocity can be derived in detail as follows. The derivation relies on a fundamental property of the cofactor matrix, for any fixed row index $i$,
\begin{align*}
\sum_{j=1}^3
\frac{\partial}{\partial\hat{x}_j}[(\det\bbJ_{\BM_h})(\bbJ_{\BM_h}^{-1})_{ji}]=0.
\end{align*}
Using this identity, we obtain
\begin{align*}
&\frac{1}{\det\bbJ_{\BM_h}}\wh\Div((\det\bbJ_{\BM_h})\hBv_h\bbJ_{\BM_h}^{-1})\nn\\
= &\frac{1}{\det\bbJ_{\BM_h}}\sum_{i,j=1}^3\frac{\partial}{\partial\hat{x}_j}
((\det\bbJ_{\BM_h})\hat{v}_i(\bbJ_{\BM_h}^{-1})_{ji})\nn\\
= &\frac{1}{\det\bbJ_{\BM_h}}\sum_{i=1}^3\sum_{j=1}^3
\frac{\partial}{\partial\hat{x}_j}[(\det\bbJ_{\BM_h})(\bbJ_{\BM_h}^{-1})_{ji}]\hat{v}_i
+ \frac{1}{\det\bbJ_{\BM_h}}\sum_{i,j=1}^3
((\det\bbJ_{\BM_h})\frac{\partial\hat{v}_i}{\partial\hat{x}_j}(\bbJ_{\BM_h}^{-1})_{ji})\nn\\
= &\sum_{i,j=1}^3\frac{\partial\hat{v}_i}{\partial\hat{x}_j}(\bbJ_{\BM_h}^{-1})_{ji}
= \mathrm{tr}(\wh\nabla\hBv_h\bbJ_{\BM_h}^{-1})=\Div\Bv_h.
\end{align*}
\end{remark}

Next, in order to prove the discrete inf-sup condition on $\DOmega_h$, we firstly present the norm equivalence.
\begin{lemma}\label{lem:equi_u}
There exist constants $C_1,C_2>0$, such that for any $\Bv_h\in\BV_{h,0}^{k}$ and $\hBv_h\in\wh\BV_{h,0}^{k}$
the following inequalities hold
\begin{align}\label{lem:eq_equi_u}
C_1\|\hBv_h\|_{\BH^1(\wh\DOmega_h)}^2\leq \|\Bv_h\|_{\BH^1(\DOmega_h)}^2\leq C_2\|\hBv_h\|_{\BH^1(\wh\DOmega_h)}^2.
\end{align}
\end{lemma}
\begin{proof}
Recall that $\Bv_h = \hBv_h\circ\BM_h^{-1}$. {By Assumptions (A4)-(A5), the mapping $\BM_h$ satisfies
\begin{align*}
\|\bbJ_{\BM_h}^{-1}\|_{\BL^{\infty}(\wh\DOmega_h)} \lesssim 1,\;\;
\delta<\det\bbJ_{\BM_h} \lesssim 1.
\end{align*}
We transform the $\BH^1$-norm of $\Bv_h$ from $\DOmega_h$ to $\wh\DOmega_h$:
\begin{align}
\|\Bv_h\|_{\BH^1(\DOmega_h)}^2
= &\int_{\DOmega_h}(|\Bv_h|^2 + |\nabla\Bv_h|^2)\D\Bx\nn\\
= &\int_{\wh\DOmega_h}(|\hBv_h|^2 + |\wh\nabla\hBv_h\,\bbJ_{\BM_h}^{-1}|^2)(\det\bbJ_{\BM_h})\D\hBx.
\end{align}
For the upper bound, using $\det\bbJ_{\BM_h}\lesssim 1$ and $\|\bbJ_{\BM_h}^{-1}\|_{\BL^\infty(\wh\DOmega_h)}\lesssim 1$, we have
\begin{align}\label{upper-bound}
\|\Bv_h\|_{\BH^1(\DOmega_h)}^2
\lesssim \|\hBv_h\|_{\BL^2(\wh\DOmega_h)}^2
+ \|\wh\nabla\hBv_h\|_{\BL^2(\wh\DOmega_h)}^2
\le C_2\|\hBv_h\|_{\BH^1(\wh\DOmega_h)}^2.
\end{align}
For the lower bound, note that $\det\bbJ_{\BM_h}\ge\delta$. For any row vector $\xi\in\mathbb{R}^3$, the inequality
\[|\xi \bbJ_{\BM_h}^{-1}| \ge \|\bbJ_{\BM_h}\|_{\BL^\infty(\wh\DOmega_h)}^{-1}|\xi|,\]
holds for invertible matrix $\bbJ_{\BM_h}$.
Together with $\|\bbJ_{\BM_h}\|_{\BL^\infty(\hat{\Omega}_h)}\le C$ from \eqref{eq:property_1st:Mh}, we obtain
\begin{align}\label{lower-bound}
\|\Bv_h\|_{\BH^1(\DOmega_h)}^2
&\ge \delta \|\hBv_h\|_{\BL^2(\wh\DOmega_h)}^2
+ \delta C^{-2} \|\wh{\nabla}\hBv_h\|_{\BL^2(\wh\DOmega_h)}^2\nn\\
&\ge C_1 \|\hBv_h\|_{\BH^1(\wh\DOmega_h)}^2,
\end{align}
where $C_1=\min\{\delta,\delta C^{-2}\}>0$.

Combining \eqref{upper-bound} and \eqref{lower-bound} completes the proof.}
\end{proof}

Similar to the proof of Lemma \ref{lem:equi_u}, we establish the following results.
\begin{lemma}
There exist constants $C_3,C_4,C_5,C_6,C_7,C_8>0$ such that for any $q_h\in Q_h^{k-1},\Bd_h\in \BD_h^{k-1},\varphi_h\in S_h^{k-2}$ and $\hat{q}_h\in \wh Q_h^{k-1},\hBd_h\in \wh \BD_h^{k-1},\hat{\varphi}_h\in \wh S_h^{k-2}$
\begin{align}
\label{lem:eq_equi_p}
&C_3\|\hat{q}_h\|_{L^2(\wh\DOmega_h)}^2
\leq\|q_h\|_{L^2(\DOmega_h)}^2
\leq C_4\|\hat{q}_h\|_{L^2(\wh\DOmega_h)}^2,\\
\label{lem:eq_equi_j}
&C_5\|\hBd_h\|_{\BH(\wh\Div,\wh\DOmega_h)}^2
\leq\|\Bd_h\|_{\BH(\Div,\DOmega_h)}^2
\leq C_6\|\hBd_h\|_{\BH(\wh\Div,\wh\DOmega_h)}^2,\\
\label{lem:eq_equi_phi}
&C_7\|\hat{\varphi}_h\|_{L^2(\wh\DOmega_h)}^2
\leq\|\varphi_h\|_{L^2(\DOmega_h)}^2
\leq C_8\|\hat{\varphi}_h\|_{L^2(\wh\DOmega_h)}^2.
\end{align}
\end{lemma}

{\begin{lemma}\label{lem:infsup_up_omega_h}
For sufficiently small $h$, there exists a constant $\beta_5>0$, such that for any $q_h\in Q_h^{k-1}$,
\begin{align}
\label{infsup_dis_up}
\sup_{\bm{0}\neq\Bv_h\in\BV_{h,0}^k}\frac{b_q(q_h,\Bv_h)_h}{\|\Bv_h\|_{\BH^1(\DOmega_h)}}
\geq \beta_5\|q_h\|_{L^2(\DOmega_h)}.
\end{align}
\end{lemma}
\begin{proof}
In the following proof, we assume that $b_q(q_h,\Bv_h)_h \geq 0$, otherwise one can reset the integrand by $-\Bv_h$.
Set $q_h=\hat  q_h\circ\BM_h^{-1}$ and $\Bv_h=\hBv_h\circ\BM_h^{-1}$. Using the transformations in Remark~\ref{rem:divuh},
\begin{align}\label{eq:bq}
&b_q(q_h,\Bv_h)_h
=\int_{\DOmega_h}q_h\,\Div\Bv_h\,\D\Bx\nn\\
=&\int_{\wh\DOmega_h}
\hat q_h\,\mathrm{tr}(\wh\nabla\hBv_h\bbJ_{\BM_h}^{-1})(\det\bbJ_{\BM_h}-1+1)\,\D\hBx\nn\\
=&\int_{\wh\DOmega_h}
\hat q_h\,\mathrm{tr}(\wh\nabla\hBv_h\bbJ_{\BM_h}^{-1})\,\D\hBx +
\int_{\wh\DOmega_h}
\hat q_h\,\mathrm{tr}(\wh\nabla\hBv_h\bbJ_{\BM_h}^{-1})(\det\bbJ_{\BM_h}-1)\,\D\hBx
\nn\\
=&\int_{\wh\DOmega_h}
\hat q_h\,\mathrm{tr}(\wh\nabla\hBv_h\bbJ_{\BM_h}^{-1}-\wh\nabla\hBv_h
+\wh\nabla\hBv_h)\,\D\hBx +
\int_{\wh\DOmega_h}
\hat q_h\,\mathrm{tr}(\wh\nabla\hBv_h\bbJ_{\BM_h}^{-1})(\det\bbJ_{\BM_h}-1)\,\D\hBx
\nn\\
=&\int_{\wh\DOmega_h}\hat q_h\,\wh\Div\hBv_h\,\D\hBx+\int_{\wh\DOmega_h}
\hat q_h\,\mathrm{tr}(\wh\nabla\hBv_h(\bbJ_{\BM_h}^{-1}-\bbI))\,\D\hBx \nn\\
&+\int_{\wh\DOmega_h}
\hat q_h\,\mathrm{tr}(\wh\nabla\hBv_h\bbJ_{\BM_h}^{-1})(\det\bbJ_{\BM_h}-1)\,\D\hBx.
\end{align}
From \eqref{eq:property_1st:Mh}, we can have
\[
\|\det\bbJ_{\BM_h}-\bm{1}\|_{\BL^{\infty}(\wh\DOmega_h)}\lesssim h,\qquad\|\bbJ_{\BM_h}^{-1}-\bbI\|_{\BL^{\infty}(\wh\DOmega_h)}\lesssim h.
\]
Using the definition of trace, the above estimate and Cauchy-Schwarz inequalities,
there exists constant $C_*$ such that,
\begin{align}\label{eq:trace}
&\left|\int_{\wh\DOmega_h}
\hat q_h\,\mathrm{tr}(\wh\nabla\hBv_h(\bbJ_{\BM_h}^{-1}-\bbI))\,\D\hBx +
\int_{\wh\DOmega_h}
\hat q_h\,\mathrm{tr}(\wh\nabla\hBv_h\bbJ_{\BM_h}^{-1})(\det\bbJ_{\BM_h}-1)\,\D\hBx\nn\right|\\
\leq&C_* h \|\hat q_h\|_{L^2(\wh\DOmega_h)}\|\hBv_h\|_{\BH^1(\wh\DOmega_h)},
\end{align}
From \eqref{lem:eq_equi_u} and \eqref{lem:eq_equi_p}, note that the following norm equivalence estimates hold,
\[\sqrt{C_2}\|\hBv_h\|_{\BH^1(\wh\DOmega_h)}\geq \|\Bv_h\|_{\BH^1(\DOmega_h)},\qquad
\|\hat q_h\|_{L^2(\wh\DOmega_h)}\geq \frac{1}{\sqrt{C_4}}\|q_h\|_{L^2(\DOmega_h)},\]
Because $b_q(q_h,\Bv_h)_h \geq 0$ holds. Inserting \eqref{eq:trace} into \eqref{eq:bq} yields
\begin{align*}
\frac{b_q(q_h,\Bv_h)_h}{\|\Bv_h\|_{\BH^1(\DOmega_h)}}
\geq &\frac{1}{\sqrt{C_2}} \frac{b_q(q_h,\Bv_h)_h}{\|\hBv_h\|_{\BH^1(\wh\DOmega_h)}}\nn\\
\geq &\frac{1}{\sqrt{C_2}}\frac{\int_{\wh\DOmega_h}\hat q_h\,\wh\Div\hBv_h\,\D\hBx}{\|\hBv_h\|_{\BH^1(\wh\DOmega_h)}}
-\frac{C_*}{\sqrt{C_2}} h\|\hat q_h\|_{L^2(\wh\DOmega_h)}.
\end{align*}
Using \eqref{infsup_dis_up_hat} and taking the supremum over $\Bv_h\in\BV_{h,0}^k$, we have
\begin{equation}
\sup_{\bm{0}\neq\Bv_h\in\BV_{h,0}^k}\frac{b_q(q_h,\Bv_h)_h}{\|\Bv_h\|_{\BH^1(\DOmega_h)}}
\geq \frac{1}{\sqrt{C_2}}  \left(\beta_3 - C_* h\right)\|\hat{q}_h\|_{L^2(\wh\DOmega_h)}.
\end{equation}
Choosing $h\le\frac{\beta_3}{2C_*}$, we obtain
\begin{equation}
\sup_{\bm{0}\neq\Bv_h\in\BV_{h,0}^k}\frac{b_q(q_h,\Bv_h)_h}{\|\Bv_h\|_{\BH^1(\DOmega_h)}}
\geq \frac{\beta_3}{2\sqrt{C_2}} \|\hat{q}_h\|_{L^2(\wh\DOmega_h)}
\geq \frac{\beta_3}{2\sqrt{C_2C_4}}\|q_h\|_{L^2(\DOmega_h)} .
\end{equation}
Letting $\beta_5 := \frac{\beta_3}{2\sqrt{C_2C_4}}$ the proof is completed.
\end{proof}
}
{
\begin{lemma}\label{lem:infsup_Jphi_omega_h}
There exists a constant $\beta_6>0$, independent of $h$, such that for any $\varphi_h\in S_h^{k-2}$,
\begin{align}
\label{infsup_dis_Jphi}
\sup_{\bm{0}\neq\Bd_h\in\BD_h^{k-1}}\frac{b_\varphi(\varphi_h,\Bd_h)_h}{\|\Bd_h\|_{\BH(\Div,\DOmega_h)}}
\geq \beta_6\|\varphi_h\|_{L^2(\DOmega_h)}>0.
\end{align}
\end{lemma}
\begin{proof}
For $\varphi_h\in S_h^{k-2}$, set $\hat\varphi_h=\varphi_h\circ \BM_h\in \wh S_h^{k-2}$. For any $\Bd_h\in\BD_h^{k-1}$, define $\hBd_h \in \widehat{\BD}_h^{k-1}$ such that
$$\Bd_h=(\det\bbJ_{\BM_h})^{-1}\bbJ_{\BM_h}\hBd_h\circ\BM_h^{-1}.$$
By the Piola's transformation identity \eqref{eq:trans_div_d},
\[
b_\varphi(\varphi_h,\Bd_h)_h
=\int_{\DOmega_h}\varphi_h\,\Div\Bd_h\,\D\Bx
=\int_{\wh\DOmega_h}\hat\varphi_h\,\wh\Div\hBd_h\,\D\hBx
=b_\varphi(\hat\varphi_h,\hBd_h)_{\hat h}.
\]
Using the inf-sup condition \eqref{infsup_dis_Jphi_hat}, we have
\begin{align*}
\sup_{\bm{0}\neq\Bd_h\in\BD_h^{k-1}}\frac{b_\varphi(\varphi_h,\Bd_h)_h}{\|\Bd_h\|_{\BH(\Div,\DOmega_h)}}
=&\sup_{\bm{0}\neq\Bd_h\in\BD_h^{k-1}}\frac{b_\varphi(\hat\varphi_h,\hBd_h)_{\hat h}}{\|\hBd_h\|_{\BH(\wh\Div,\wh\DOmega_h)}}
\cdot\frac{\|\hBd_h\|_{\BH(\wh\Div,\wh\DOmega_h)}}{\|\Bd_h\|_{\BH(\Div,\DOmega_h)}}\\
\geq & \beta_4\|\hat\varphi_h\|_{L^2(\wh\DOmega_h)}\sup_{\bm{0}\neq\Bd_h\in\BD_h^{k-1}}
\frac{\|\hBd_h\|_{\BH(\wh\Div,\wh\DOmega_h)}}{\|\Bd_h\|_{\BH(\Div,\DOmega_h)}}
\end{align*}
Applying the norm equivalence estimates \eqref{lem:eq_equi_j} and \eqref{lem:eq_equi_phi},
\[
\|\hBd_h\|_{\BH(\wh\Div,\wh\DOmega_h)}\geq \frac{1}{\sqrt{C_6}}\|\Bd_h\|_{\BH(\Div,\DOmega_h)},\qquad
\|\hat\varphi_h\|_{L^2(\wh\DOmega_h)}\geq \frac{1}{\sqrt{C_8}}\|\varphi_h\|_{L^2(\DOmega_h)},
\]
we obtain
\[
\sup_{\bm{0}\neq\Bd_h\in\BD_h^{k-1}}\frac{b_\varphi(\varphi_h,\Bd_h)_h}{\|\Bd_h\|_{\BH(\Div,\DOmega_h)}}
\geq \frac{\beta_4}{\sqrt{C_6C_8}}
\|\varphi_h\|_{L^2(\DOmega_h)}.
\]
Taking $\beta_6 := \frac{\beta_4}{\sqrt{C_6C_8}}$ completes the proof.
\end{proof}
\begin{remark}
The discrete inf-sup condition \eqref{infsup_dis_Jphi} also guarantees the uniqueness of $\phi_h$; no additional normalization is required, as the boundary condition is already enforced weakly through the variational formulation.
\end{remark}
}

Now we show the finite element discretization form of \eqref{eq:model_3st}: find $(\Bu_h,\BJ_h)\in\BV_{h,0}^k(\Div 0)\times\BD_h^{k-1}(\Div 0)$ such that
\begin{align}\label{eq:model_7st}
\Ca(\Bu_h,\BJ_h;\Bv_h,\Bd_h)_h + \Co(\Bu_h;\Bu_h,\Bv_h)_h = \Cl(\Bf_h,\Bg_h;\Bv_h,\Bd_h)_h,
\end{align}
for all $(\Bv_h,\Bd_h)\in\BV_{h,0}^k(\Div 0)\times\BD_h^{k-1}(\Div 0)$.

Similar to the continuous form, we define
\begin{align*}
&|||(\Bv,\Bd)|||_{\Ca_h}^2 = \|\Bv\|_{\BH^1(\DOmega_h)}^2 + \|\Bd\|_{\BH(\Div,\DOmega_h)}^2,\\
&|||(\Bf,\Bg)|||_{\Cf_h}^2 = \|\Bf\|_{\BL^2(\DOmega_h)}^2 + \|\Bg\|_{\BL^2(\DOmega_h)}^2,\\
&|||\Cl|||_{h} = \sup\limits_{(\bm{0},\bm{0})\neq(\Bv_h,\Bd_h)\in\BV_{h,0}^k(\Div 0)\times\BD_h^{k-1}(\Div 0)}\frac{\Cl(\Bf_h,\Bg_h;\Bv_h,\Bd_h)_h}{|||(\Bv_h,\Bd_h)|||_{\Ca_h}},
\end{align*}
and it is obvious that $|||\Cl|||_{h}\leq\max\{1,\alpha\}|||(\Bf,\Bg)|||_{\Cf_h}$ follows from $\Bf_h = \Bf\circ\Phi_h^{-1},\Bg_h = \Bg\circ\Phi_h^{-1}$.
\begin{theorem}\label{thm_div0_exist_dis}
For $\Bf_h, \Bg_h\in\BL^2(\DOmega_h)$, there exists at least one solution of the variational problem \eqref{eq:model_7st}. We have the stability bound
\ben
|||(\Bu_h,\BJ_h)|||_{\Ca_h}\leq\frac{|||\Cl|||_{h}}{C_a\min\{R_e^{-1},\alpha\}}.
\een
Furthermore, for small data with
\begin{align}
\label{div0_unique_dis}
\frac{C_{\Co}|||\Cl|||_{h}}{C_a^2\min\{R_e^{-2},\alpha^2\}} \leq 1,
\end{align}
the variational problem \eqref{eq:model_7st} has a unique solution.
\end{theorem}
\begin{proof}
The proof is similar to that of Theorem \ref{thm_div0_exist_con}.
\qed
\end{proof}
\subsection{Parametric mixed finite element approximation}
We propose the following discrete variational problem of \eqref{eq:model_5st} to solve inductionless MHD equations:
Find
$(\Bu_h,p_h, \BJ_h, \phi_h) \in \BV_{h,0}^k\times Q_h^{k-1}\times\BD_h^{k-1}\times S_h^{k-2}$, such that
\begin{subequations}\label{eq:model_8st-0}
\begin{align}
\Co(\Bu_h;\Bu_h,\Bv_h)_h + R_e^{-1}(\nabla\Bu_h,\nabla\Bv_h)_h + \gamma(\Div\Bu_h,\Div\Bv_h)_h\nn\\
 - (p_h,\Div\Bv_h)_h- \alpha(\BJ_h\times\BB,\Bv_h)_h = (\Bf_h,\Bv_h)_h,\label{eq:model_8st-0:1}\\
(\BJ_h,\Bd_h)_h - (\phi_h,\Div\Bd_h)_h - (\Bu_h\times\BB,\Bd_h)_h = (\Bg_h,\Bd_h)_h,\label{eq:model_8st-0:2}\\
-(\Div{\Bu_h},q_h)_h = 0,\label{eq:model_8st-0:3}\\
-(\Div{\BJ_h},\varphi_h)_h = 0.\label{eq:model_8st-0:4}
\end{align}
\end{subequations}
holds for any  $(\Bv_h,q_h,\Bd_h,\varphi_h)\in\BV_{h,0}^k\times Q_h^{k-1}\times\BD_h^{k-1}\times S_h^{k-2}$.

The compact form is given by
\begin{subequations}\label{eq:model_8st}
\begin{align}
\Ca(\Bu_h,\BJ_h;\Bv_h,\Bd_h)_h + \gamma(\Div\Bu_h,\Div\Bv_h)_h + \Co(\Bu_h;\Bu_h,\Bv_h)_h\nn\\
- \Cb(p_h,\phi_h;\Bv_h,\Bd_h)_h = \Cl(\Bf_h,\Bg_h;\Bv_h,\Bd_h)_h,\label{eq:model_8st:1}\\
\Cb(q_h,\varphi_h;\Bu_h,\BJ_h)_h = 0,\label{eq:model_8st:2}
\end{align}
\end{subequations}
\vspace{5mm}

\begin{lemma}\label{lemma_B_infsup_dis}
There is a constant $\Gamma_h>0$ only depending on $\Omega_h$ such that
\begin{align}\label{eq:lemma_B_infsup_dis}
\sup_{(\Bv_h,\Bd_h)\in\BV_{h,0}^k\times\BD_h^{k-1}}\frac{\Cb(q_h,\varphi_h;\Bv_h,\Bd_h)_h}
{|||(\Bv_h,\Bd_h)_h|||_{\Ca_h}}
\geq\Gamma_h|||(q_h,\varphi_h)|||_{\Cf_h}.
\end{align}
for all $(q_h,\varphi_h)\in Q_h^{k-1}\times S_h^{k-2}$.
\end{lemma}
\begin{proof}
The proof is similar to that of Lemma \ref{lemma_B_infsup_con}.
\qed
\end{proof}
\vspace{5mm}

\begin{lemma}\label{lem:div_J0}
The solution $\BJ_h$ of the discrete variational problem \eqref{eq:model_8st-0} is exactly divergence-free, namely $\Div\BJ_h=0$ on $\DOmega_h$.
\end{lemma}
\begin{proof}
The definitions of spaces $S_h^{k-2}$ allow for an appropriate selection
\begin{align*}
&\varphi_h=\det\bbJ_{\BM_h}\Div\BJ_h = \wh\Div\hBJ_h\circ\BM_h^{-1}\in S_h^{k-2},
\end{align*}
where $\hBJ_h = (\det\bbJ_{\BM_h})\bbJ_{\BM_h}^{-1}\BJ_h$.
Using the fact
\begin{align*}
&0 = (\Div\BJ_h,\varphi_h)_h
= (\Div\BJ_h,\det\bbJ_{\BM_h}\Div\BJ_h)_h
> \delta \|\Div\BJ_h\|_{L^2(\DOmega_h)}^2,
\end{align*}
we conclude that $\Div\BJ_h=0$.
\qed
\end{proof}
\begin{theorem}
For $\Bf_h, \Bg_h\in\BL^2(\DOmega_h)$, there exists at least one solution of the mixed variational problem \eqref{eq:model_8st}. We have the stability bounds
\begin{align*}
|||(\Bu_h,\BJ_h)|||_{\Ca_h}\leq \frac{|||\Cl|||_{h}}{C_a\min\{R_e^{-1},\alpha\}}
\end{align*}
and
\begin{align*}
|||(p_h,\phi_h)|||_{\Cf_h}
\leq &\Gamma_h^{-1}\big[\max\{R_e^{-1},\alpha, \alpha {\|\BB\|_{\BL^\infty(\DOmega_h)}}, \gamma\}|||(\Bu_h,\BJ_h)|||_{\Ca_h}\\
&+ C_{\Co}|||(\Bu_h,\BJ_h)|||_{\Ca_h} + \max\{1,\alpha\}|||(\Bf_h,\Bg_h)|||_{\Cf_h}\big]
\end{align*}
for any solution $(\Bu_h,p_h,\BJ_h,\phi_h)$. Further, under the assumption \eqref{div0_unique_dis}, there is a unique solution of \eqref{eq:model_8st}.
\end{theorem}
\begin{proof}
The proof is similar to that of Theorem \ref{thm_div0_exist_con}.
\end{proof}

\section{A priori Error estimates}
\label{sec:error}
Since the computational domain $\DOmega_h$ usually does not exactly coincide with the physical domain $\DOmega$  (despite sharing the same vertices), we will employ the regular extension assumption to obtain the finite element estimates. We then define the hold-all domain $D_H$ as the convex hull of their union, i.e., $D_H:=\mathrm{conv}(\DOmega\cup\DOmega_h)$ \cite{rc23}. Specifically, we assume that the exact solution $\Bu, p, \BJ$ and $\phi$ admit smooth extension to $D_H$, with $\Bu$ and $\BJ$ being divergence-free in $D_H$. The magnetic field $\BB$, as a given quantity, is naturally extended to $\BL^{\infty}(D_H)$. The source terms $\Bf$ and $\Bg$ are extended to $D_H$ by primal equations \eqref{eq:model_2-0st}. Consequently, the extended true function $\tBu, \tilde{p}, \tBJ, \tilde{\phi}$, the magnetic filed $\tilde{\BB}$ and the source terms
$\tBf, \tBg$ are well-defined on the computational domain $\DOmega_h$.
More precisely, we assume
\begin{align*}
(\tilde{\Bu},\tilde{p}, \tilde{\BJ}, \tilde{\phi})&\in \BH^{k+1}(D_H)\times H^k(D_H)\times\BH^{k}(D_H)\times H^{k-1}(D_H), k\geq1,\\
(\tilde{\Bf},\tilde{\Bg})&\in\BW^{1,\infty}(D_H)\times\BW^{1,\infty}(D_H),
\end{align*}
with $\Div\tilde{\Bu} = 0$ and $\Div\tilde{\BJ} = 0$ holding on $D_H$.

\subsection{Preliminaries}
For the subsequent error analysis, we introduce the following lemma.
\begin{lemma}\label{lem:u-u_sharp}
Based on $\Phi_h=\BM_h\circ\BM^{-1}$, we define $\Bv^\sharp = \Bv\circ\Phi_h^{-1}$ on $\DOmega_h$. Let $\tBv\in\BH^2(D_H)$ be the extension of $\Bv$ to $D_H$, the following estimate holds
\begin{align*}
\|\tBv-\Bv^\sharp\|_{\BH^1(\DOmega_h)} \lesssim h^k\|\tBv\|_{\BH^2(D_H)}.
\end{align*}
\end{lemma}
\begin{proof}
If $\Bv = 0$ on $\partial\DOmega$, then
\begin{align*}
\Bv^\sharp(\Bx) = \Bv(\Phi_h^{-1}(\Bx)) = \bm{0},\quad \textrm{on}\;\;\partial\DOmega_h.
\end{align*}
Using the definition,
\begin{align*}
\|\tBv-\Bv^\sharp\|_{\BL^2(\DOmega_h)}^2
&= \int_{\DOmega_h}|\tBv\circ(\mathrm{Id}-\Phi_h^{-1})|^2\D\Bx
\lesssim \|\mathrm{Id}-\Phi_h\|_{\BL^\infty(\DOmega_h)}^2\|\tBv\|_{\BH^1(\DOmega_h)}^2.
\end{align*}
Through the chain rule, we have
\begin{align}
&\|\nabla(\tBv-\Bv^\sharp)\|_{\BL^2(\DOmega_h)}^2\nn\\
=& \int_{\DOmega_h}|\nabla\tBv-(\nabla\tBv\circ\Phi_h^{-1})\bbJ_{\Phi_h}^{-1}|^2\D\Bx\nn\\
=& \int_{\DOmega_h}|\nabla\tBv(\bbI-\bbJ_{\Phi_h}^{-1}) + \nabla\tBv\circ(\mathrm{Id}-\Phi_h^{-1})\bbJ_{\Phi_h}^{-1}|^2\D\Bx\nn\\
\lesssim& \|\nabla\tBv(\bbI-\bbJ_{\Phi_h}^{-1})\|_{\BL^2(\DOmega_h)}^2 + \| \nabla\tBv\circ(\mathrm{Id}-\Phi_h^{-1})\|_{\BL^2(\DOmega_h)}^2\nn\\
\lesssim& \|\nabla\tBv\|_{\BL^2(\DOmega_h)}^2\|\bbI-\bbJ_{\Phi_h}^{-1}\|_{\BL^\infty(\DOmega_h)}^2 + \|\mathrm{Id}-\Phi_h^{-1}\|_{\BL^\infty(\DOmega_h)}^2\| \nabla\tBv\|_{\BH^1(\DOmega_h)}^2
\end{align}
where $\bbJ_{\Phi_h}^{-1}$ is the Jacobian matrix of $\Phi_h^{-1}$.

Using the inequality
\begin{align}\label{pro_Phi}
\|\mathrm{Id}-\Phi_h^{-1}\|_{\BL^\infty(\DOmega_h)}
\lesssim h^{k+1},\quad\|\bbI-\bbJ_{\Phi_h}^{-1}\|_{\BL^\infty(\DOmega_h)}\lesssim h^k
\end{align}
from \cite{fg16,ml86}, we have
\begin{align}
\|\tBv-\Bv^\sharp\|_{\BH^1(\DOmega_h)}^2
= &\|\tBv-\Bv^\sharp\|_{\BL^2(\DOmega_h)}^2 + \|\nabla(\tBv-\Bv^\sharp)\|_{\BL^2(\DOmega_h)}^2\nn\\
\lesssim &h^{2k+2}\|\tBv\|_{\BH^1(\DOmega_h)}^2 + h^{2k}\|\tBv\|_{\BH^2(\DOmega_h)}^2
\lesssim h^{2k}\|\tBv\|_{\BH^2(D_H)}^2.
\end{align}
The proof is completed.
\qed
\end{proof}

Similarly, for $\phi$, we obtain the following lemma.
\begin{lemma}\label{lem:phi-phi_sharp}
Define $\phi^\sharp = \phi\circ\Phi_h^{-1}$ on $\DOmega_h$. Let $\tilde{\phi}\in H^2(D_H)$ be the extension of $\phi$ to $D_H$, the following estimate holds
\begin{align*}
\|\tilde{\phi} - \phi^\sharp\|_{H^1(\DOmega_h)} \lesssim h^k\|\tilde{\phi}\|_{H^2(D_H)}.
\end{align*}
\end{lemma}
\begin{remark}
In the subsequent discussion, we let $\Bu, p, \BJ, \phi, \BB, \Bf$ and $\Bg$ refer to their extended forms $\tBu, \tilde{p}, \tBJ, \tilde{\phi}, \tilde{\BB}, \tBf$ and $\tBg$, respectively.
\end{remark}

Following the method of Theorem 2 in \cite{fg16}, we obtain the following lemma.
\begin{lemma}\label{lem:proj_J}
For the parametric BDM space $\BD_h^{k-1}\subset\BH(\Div,\DOmega_h)$, there exists an interpolation operator $\Cr_d:\BH(\Div,\DOmega_h)\cap\BL^s(\DOmega_h)\rightarrow \BD_h^{k-1}$ for some $s>0$, such that for $\BJ\in\BH^k(\DOmega_h)\cap\BH(\Div0,\DOmega_h)$, we have the following estimate
\begin{align*}
\|\BJ-\Cr_d\BJ\|_{\BH(\Div,\DOmega_h)} = \|\BJ-\Cr_d\BJ\|_{\BL^2(\DOmega_h)}
\lesssim h^k \|\BJ\|_{\BH^{k}(D_H)}.
\end{align*}
\end{lemma}
\begin{proof}
The construction of $\Cr_d$ is based on the standard BDM interpolation operator $\wh\Cr_d:\BH(\Div,\wh\DOmega_h)\cap \wh\BL^s(\DOmega_h)\rightarrow \wh\BD_h^{k-1}$ on the reference polyhedral domain $\wh\DOmega_h$ \cite{MFEMbook}. For $\BJ\in\BH(\Div 0,\DOmega_h)$, define its Piola's transform on the reference domain by
\begin{align*}
\hBJ(\hBx) = (\det{\bbJ}_{\BM_h}(\hBx)){\bbJ}_{\BM_h}(\hBx)^{-1}\BJ(\BM_h(\hBx)).
\end{align*}
The Piola's transform preserves the property of divergence-free, indeed
\begin{align*}
\wh\Div\hBJ(\hBx) = (\det\bbJ_{\BM_h}(\hBx))\Div\BJ(\BM_h(\hBx))=0.
\end{align*}
Now define the interpolation operator as
\begin{align*}
(\Cr_d\BJ)(\BM_h(\hBx)) = \frac{1}{\det{\bbJ}_{\BM_h}(\hBx)}{\bbJ}_{\BM_h}(\hBx)(\wh\Cr_d\hBJ)(\hBx).
\end{align*}
Since the reference operator $\wh\Cr_d$ maps into a divergence-free discrete space and the Piola's transform preserves normal continuity, we obtain $\Cr_d\BJ\in\BD_h^{k-1}(\Div0)$. Consequently, we obtain
\begin{align}
\|\BJ-\Cr_d\BJ\|_{\BH(\Div,\DOmega_h)} &= \|\BJ-\Cr_d\BJ\|_{\BL^2(\DOmega_h)}\nn\\
&= \|(\hBJ - \wh\Cr_d\hBJ)\bbJ_{\BM_h}(\det\bbJ_{\BM_h})^{-1/2}\|_{\BL^2(\wh\DOmega_h)}.
\end{align}
Applying the estimate of \eqref{eq:property_1st:Mh}, it is straightforward to obtain the result.
\qed
\end{proof}
\vspace{5mm}

\begin{lemma}\label{lem:proj_u}
For the isoparametric finite element spaces $\BV_{h,0}^k\subset\BH_0^1(\DOmega_h)$, there exists an interpolation operator $\Cr_v:\BH_0^{k+1}(\DOmega_h)\rightarrow\BV_{h,0}^k$ such that
\begin{align*}
\|\Bu-\Cr_v\Bu\|_{\BH^1(\DOmega_h)}\lesssim h^k\|\Bu\|_{\BH^{k+1}(D_H)}.
\end{align*}
holds for all $\Bu\in\BH_0^{k+1}(\DOmega_h)$.
\end{lemma}
\begin{proof}
The construction method of the interpolation operator is similar to that in Lemma \ref{lem:proj_J}. Through isoparametric mapping,
\begin{align*}
\hBu(\hBx)=\Bu(\BM_h(\hBx)),\;\;\;\hBu|_{\wh\partial\DOmega_h}=0.
\end{align*}
The interpolation operator $\Cr_v$ is defined as
\begin{align*}
(\Cr_v\Bu)(\BM_h(\hBx))=\wh\Cr_v\hBu(\hBx)
\end{align*}
where $\wh\Cr_v\hBu$ is the standard Lagrange interpolation of $\hBu$ on the straight-edged tetrahedral mesh. Since $\wh\Cr_v\hBu\in\wh\BV_{h,0}^k$, it follows that $\Cr_v\Bu\in\BV_{h,0}^k$.

Based on the definition,
\begin{align}\label{eq:lem_2st_1st}
\|\Bu-\Cr_v\Bu\|_{\BH^1(\DOmega_h)}
= \|\Bu-\Cr_v\Bu\|_{\BL^2(\DOmega_h)} + \|\nabla(\Bu-\Cr_v\Bu)\|_{\BL^2(\DOmega_h)}.
\end{align}
Using the standard Lagrange interpolation operator leads to
\begin{align}\label{eq:lem_2st_2st}
\|\Bu - \Cr_v\Bu\|_{\BL^2(\DOmega_h)}=\|(\hBu - \wh\Cr_v\hBu)(\det\bbJ_{\BM_h})^{1/2}\|_{\BL^2(\wh\DOmega_h)}.
\end{align}
By the relationship of \eqref{eq:trans_grad_v} gives
\begin{align}\label{eq:lem_2st_3st}
\|\nabla(\Bu-\Cr_v\Bu)\|_{\BL^2(\DOmega_h)}
= &\|(\wh\nabla\hBu-\wh\nabla\wh\Cr_v\hBu)\bbJ_{\BM_h}^{-1}(\det\bbJ_{\BM_h})^{1/2}\|_{\BL^2(\wh\DOmega_h)}\nn\\
\lesssim &\|\wh\nabla\hBu-\wh\nabla\wh\Cr_v\hBu\|_{\BL^2(\wh\DOmega_h)}.
\end{align}
Substituting \eqref{eq:lem_2st_2st} and \eqref{eq:lem_2st_3st} into \eqref{eq:lem_2st_1st} yields
\begin{align}
&\|\Bu-\Cr_v\Bu\|_{\BH^1(\DOmega_h)}\nn\\
\lesssim &\|\hBu - \wh\Cr_v\hBu\|_{\BL^2(\wh\DOmega_h)}
+ \|\wh\nabla\hBu-\wh\nabla\wh\Cr_v\hBu\|_{\BL^2(\wh\DOmega_h)}\nn\\
\lesssim & h^{k+1}\|\hBu\|_{\BH^{k+1}(\wh\DOmega_h)}+h^k\|\hBu\|_{\BH^{k+1}(\wh\DOmega_h)}
\lesssim h^k\|\Bu\|_{\BH^{k+1}(D_H)}.
\end{align}
The proof is completed.
\qed
\end{proof}
\vspace{5mm}

Using the same lines of proof and methodology, the following two lemmas can be verified.
\begin{lemma}\label{lem:proj_p}
Suppose $p\in H^k(\DOmega_h)$. Defining the $L^2$ projection operator $\Cp_q:H^k(\DOmega_h)\rightarrow Q_h^{k-1}$, the following projection error estimate holds:
\begin{align}\label{eq:lem_proj_p}
\|p-\Cp_qp\|_{L^2(\DOmega_h)}\lesssim h^{k}\|p\|_{H^{k}(D_H)}.
\end{align}
\end{lemma}
\begin{lemma}\label{lem:proj_phi}
Suppose $\phi\in H^{k-1}(\DOmega_h)$ . Defining $L^2$ projection operator $\Cp_{\varphi}: H^{k-1}(\DOmega_h)\rightarrow S_h^{k-2}$, the following projection error estimate holds:
\begin{align}\label{eq:lem_proj_phi}
\|\phi-\Cp_{\varphi}\phi\|_{L^2(\DOmega_h)}\lesssim h^{k-1}\|\phi\|_{H^{k-1}(D_H)}.
\end{align}
\end{lemma}

\subsection{Optimal error estimates in the energy-norm}
For the extended true solutions velocity and pressure $(\Bu, p)$, on the curved mesh $\Ct_h$, we define the following Stokes projection with grad-div stabilization
$(\breve{\Bu}_h, \breve{p}_h) \in \BV_{h,0}^k \times Q_h^{k-1}$ by
\begin{subequations}\label{eq:stoke_proj}
\begin{align}
R_e^{-1}(\nabla(\Bu - \breve{\Bu}_h), \nabla\Bv_h)_h + \gamma(\Div(\Bu - \breve{\Bu}_h), \Div\Bv_h)_h&\nn\\
-(p - \breve{p}_h, \Div\Bv_h)_h&=0,\\
-(q_h, \Div(\Bu - \breve{\Bu}_h))_h&=0,
\end{align}
\end{subequations}
for test function $(\Bv_h,q_h)\in\BV_{h,0}^k\times Q_h^{k-1}$.

For the current density and electric potential $(\BJ,\phi)$, we define the mixed Poisson projection $(\breve{\BJ}_h,\breve{\phi}_h)\in\BD_h^{k-1}\times S_h^{k-2}$
on the curved mesh $\Ct_h$ by
\begin{subequations}\label{proj_mix_poisson}
\begin{align}
(\BJ - \breve{\BJ}_h,\Bd_h)_h - (\phi-\breve{\phi}_h,\Div\Bd_h)_h = 0,\label{proj_mix_poisson_1}\\
(\varphi_h,\Div(\BJ-\breve{\BJ}_h))_h = 0,\label{proj_mix_poisson_2}
\end{align}
\end{subequations}
for test function $(\Bd_h,\varphi_h)\in \BD_h^{k-1}(\Div 0)\times S_h^{k-2}$.
\begin{lemma}
For the Stokes projection \eqref{eq:stoke_proj}, by applying Lemma \ref{lem:proj_u}-\ref{lem:proj_p} and Aubin-Nitsche duality argument, we establish the error estimate
\begin{align}\label{eq:stoke_proj_prop}
&\|\Bu - \breve{\Bu}_h\|_{\BL^2(\DOmega_h)} + h\left( |||\Bu - \breve{\Bu}_h|||_{\Cd} + \|p - \breve{p}_h\|_{L^2(\DOmega_h)} \right)\nn\\
\leq& {C(\gamma)}h^{k+1} \left( \|\Bu\|_{\BH^{k+1}(D_H)} + \|p\|_{H^k(D_H)} \right)
\end{align}
where
$$|||\Bv_h|||_{\Cd}^2 := R_e^{-1}\|\nabla\Bv_h\|_{\BL^2(\DOmega_h)}^2 + \gamma\|\Div\Bv_h\|_{L^2(\DOmega_h)}^2,$$
and the constant $C$ is independent of $R_e$.
\end{lemma}
\begin{proof}
For notational simplicity, let
\begin{align*}
&\Bu - \breve{\Bu}_h = \Bu - \Bu^\sharp + \Bu^\sharp - \Cr_v\Bu^\sharp - (\breve{\Bu}_h - \Cr_v\Bu^\sharp)
:= \Be_1 + \Be_2 - \Be_3,\\
&p - \breve{p}_ h= p - \Cp_qp - (\breve{p}_h - \Cp_qp) := e_4 - e_5.
\end{align*}
Substituting these notations into \eqref{eq:stoke_proj} gives
\begin{subequations}
\begin{align}
&R_e^{-1}(\nabla\Be_3, \nabla\Bv_h)_h + \gamma(\Div\Be_3,\Div\Bv_h)_h + (e_5, \Div\Bv_h)_h\nn\\
&= R_e^{-1}(\nabla(\Be_1+\Be_2), \nabla\Bv_h)_h + \gamma(\Div(\Be_1+\Be_2),\Div\Bv_h)_h- (e_4, \Div\Bv_h)_h,\label{eq:stoke_proj_1}\\
&(q_h, \Div\Be_3)_h = (q_h, \Div(\Be_1+\Be_2))_h.
\end{align}
\end{subequations}
Taking $\Bv_h=\Be_3, q_h=e_5$. Applying the Cauchy-Schwarz inequality yields
\begin{align}\label{eq:e3_error}
&R_e^{-1}\|\nabla\Be_3\|_{\BL^2(\DOmega_h)}^2 + \gamma\|\Div\Be_3\|_{L^2(\DOmega_h)}^2\nn\\
= &R_e^{-1}(\nabla(\Be_1+\Be_2), \nabla\Be_3)_h + \gamma(\Div(\Be_1+\Be_2),\Div\Be_3)_h - (e_4, \Div\Be_3)_h - (e_5, \Div\Be_3)_h\nn\\
\leq &R_e^{-1}\|\nabla(\Be_1+\Be_2)\|_{\BL^2(\DOmega_h)} \|\nabla\Be_3\|_{\BL^2(\DOmega_h)}
+ \gamma\|\Div(\Be_1+\Be_2)\|_{L^2(\DOmega_h)}\|\Div\Be_3\|_{L^2(\DOmega_h)}\nn\\
&+ \|e_4\|_{L^2(\DOmega_h)} \|\Div\Be_3\|_{L^2(\DOmega_h)} + \|e_5\|_{L^2(\DOmega_h)} \|\Div(\Be_1+\Be_2)\|_{L^2(\DOmega_h)}.
\end{align}
Applying the inf-sup condition with \eqref{eq:stoke_proj_1}, we obtain
\begin{align}\label{eq:e7_error}
&\beta_6\|e_5\|_{L^2(\DOmega_h)} \leq \frac{(e_5, \Div\Bv_h)_h}{\|\Bv_h\|_{\BH^1(\DOmega_h)}}\nn\\
\leq &\frac{R_e^{-1}(\nabla\Be_3, \nabla\Bv_h)_h + R_e^{-1}(\nabla(\Be_1+\Be_2), \nabla\Bv_h)_h + (e_4, \Div\Bv_h)_h}{\|\Bv_h\|_{\BH^1(\DOmega_h)}}\nn\\
&+ \frac{\gamma(\Div(\Be_1+\Be_2),\Div\Bv_h)_h + \gamma(\Div\Be_3,\Div\Bv_h)_h}{\|\Bv_h\|_{\BH^1(\DOmega_h)}}\nn\\
\leq& R_e^{-1}\big(\|\nabla\Be_3\|_{\BL^2(\DOmega_h)} + \|\nabla(\Be_1+\Be_2)\|_{\BL^2(\DOmega_h)}\big)
+ \|e_4\|_{L^2(\DOmega_h)}\nn\\
&+ \gamma\big(\|\Div(\Be_1+\Be_2)\|_{L^2(\DOmega_h)} + \|\Div\Be_3\|_{L^2(\DOmega_h)}\big).
\end{align}
Substituting \eqref{eq:e7_error} into \eqref{eq:e3_error} has
\begin{align}
&R_e^{-1}\|\nabla\Be_3\|_{\BL^2(\DOmega_h)}^2 + \gamma\|\Div\Be_3\|_{L^2(\DOmega_h)}^2\nn\\
\leq &R_e^{-1}\|\nabla(\Be_1+\Be_2)\|_{\BL^2(\DOmega_h)} \|\nabla\Be_3\|_{\BL^2(\DOmega_h)}
+ \gamma\|\Div(\Be_1+\Be_2)\|_{L^2(\DOmega_h)}\|\Div\Be_3\|_{L^2(\DOmega_h)}\nn\\
&+ \|e_4\|_{L^2(\DOmega_h)} \|\Div\Be_3\|_{L^2(\DOmega_h)}
+ \beta_6^{-1}\|e_4\|_{L^2(\DOmega_h)}\|\Div(\Be_1+\Be_2)\|_{L^2(\DOmega_h)}\nn\\
&+ \beta_6^{-1}\big(R_e^{-1}(\|\nabla\Be_3\|_{\BL^2(\DOmega_h)}
+ \|\nabla(\Be_1+\Be_2)\|_{\BL^2(\DOmega_h)})\nn\\
&+ \gamma\big(\|\Div(\Be_1+\Be_2)\|_{L^2(\DOmega_h)}
+ \|\Div\Be_3\|_{L^2(\DOmega_h)}\big)\|\Div(\Be_1+\Be_2)\|_{L^2(\DOmega_h)}.
\end{align}
Using the divergence inequality in three-dimensional space, we have
\begin{align}
|||\Be_3|||_{\Cd}^2
\leq &|||\Be_3|||_{\Cd}\big((1+\sqrt{3}
+ 2\sqrt{3}\beta_6^{-1})\|\nabla(\Be_1+\Be_2)\|_{\BL^2(\DOmega_h)}\nn\\
&+ \gamma^{-1}\|e_4\|_{L^2(\DOmega_h)}\big)
+ \sqrt{3}\beta_6^{-1}\|e_4\|_{L^2(\DOmega_h)}\|\nabla(\Be_1+\Be_2)\|_{\BL^2(\DOmega_h)}\nn\\
&+ \beta_6^{-1}(\sqrt{3}+3\gamma)\|\nabla(\Be_1+\Be_2)\|_{\BL^2(\DOmega_h)}^2\nn\\
\leq &\max\{1+\sqrt{3}+2\sqrt{3}\beta_6^{-1},\gamma^{-1}\}|||\Be_3|||_{\Cd}\nn\\
&\big(\|\nabla(\Be_1+\Be_2)\|_{\BL^2(\DOmega_h)}
+\|e_4\|_{L^2(\DOmega_h)}\big)\nn\\
+&\beta_6^{-1}(\sqrt{3}+3\gamma)
\big(\|e_4\|_{L^2(\DOmega_h)}+\|\nabla(\Be_1+\Be_2)\|_{\BL^2(\DOmega_h)}\big)^2.
\end{align}

Solving this quadratic inequality gives
\begin{align*}
|||\Be_3|||_{\Cd} \leq C(\beta_6,\gamma)\big(\|e_4\|_{L^2(\DOmega_h)} + \|\nabla(\Be_1+\Be_2)\|_{\BL^2(\DOmega_h)}\big).
\end{align*}
Applying the geometric error estimate in Lemma \ref{lem:u-u_sharp}, together with the results of Lemma \ref{lem:proj_u} and Lemma \ref{lem:proj_p} yields
\begin{align}\label{eq:error_e3}
|||\Be_3|||_{\Cd} \leq C(\beta_6,\gamma)h^k(\|\Bu\|_{\BH^{k+1}(D_H)}+\|p\|_{H^k(D_H)}).
\end{align}
Furthermore, combing \eqref{eq:e7_error} and \eqref{eq:error_e3}, we also have
\begin{align}
\|e_5\|_{L^2(\DOmega_h)} \leq C(\beta_6,\gamma)(\|\Bu\|_{\BH^{k+1}(D_H)}+\|p\|_{H^k(D_H)}).
\end{align}

Next, we introduce an auxiliary dual problem
\begin{subequations}
\begin{align}
-R_e^{-1}\Delta\Bw+\nabla\psi=\theta,\;\;&\textrm{in}\;\;\DOmega_h\\
\Div\Bw=0,\;\;&\textrm{in}\;\;\DOmega_h
\end{align}
\end{subequations}
with the homogeneous Dirichlet boundary condition $\Bw=0$ on $\partial\DOmega_h$. This problem admits the regularity estimate $\|\Bw\|_{\BH^2(\DOmega_h)}+\|\psi\|_{H^1(\DOmega_h)}\leq\|\theta\|_{\BL^2(\DOmega_h)}$.
Setting $\theta=\Be_3$ and taking $L^2$-inner product with $\Be_3$ leads to
\begin{align}
&\|\Be_3\|_{\BL^2(\DOmega_h)}^2\nn\\
= &R_e^{-1}(\nabla\Bw,\nabla\Be_3)_h - (\psi,\Div\Be_3)_h\nn\\
= &R_e^{-1}(\nabla(\Bw-\Bw_h),\nabla\Be_3)_h - (\psi-\psi_h,\Div\Be_3)_h\nn\\
\leq &R_e^{-1}\|\nabla(\Bw-\Bw_h)\|_{\BL^2(\DOmega_h)}\|\nabla\Be_3\|_{\BL^2(\DOmega_h)}
+ \|\psi-\psi_h\|_{L^2(\DOmega_h)}\|\Div\Be_3\|_{L^2(\DOmega_h)}\nn\\
\leq &\max\{1,\gamma^{-1}\}|||\Be_3|||_{\Cd}(h\|\Bw\|_{\BH^2(\DOmega_h)} + h\|\psi\|_{H^1(\DOmega_h)})\nn\\
\leq &\max\{1,\gamma^{-1}\}h|||\Be_3|||_{\Cd}\|\Be_3\|_{\BL^2(\DOmega_h)}.
\end{align}
Combining \eqref{eq:error_e3} yields
\begin{align*}
\|\Be_3\|_{\BL^2(\DOmega_h)}\leq C(\beta_6,\gamma)h^{k+1}(\|\Bu\|_{\BH^{k+1}(D_H)}+\|p\|_{H^k(D_H)}).
\end{align*}
Applying Lemma \ref{lem:u-u_sharp} and Lemma \ref{lem:proj_u},
the proof is completed.
\qed
\end{proof}
\vspace{5mm}

\begin{lemma}
For the mixed Poisson projection \eqref{proj_mix_poisson}, the following estimate holds
\begin{align}\label{eq:poisson_proj_prop}
\|\BJ-\breve{\BJ}_h\|_{\BL^2(\DOmega_h)} + h\|\phi-\breve{\phi}_h\|_{L^2(\DOmega_h)}
\lesssim h^k(\|\BJ\|_{\BH^k(D_H)}+\|\phi\|_{H^{k-1}(D_H)})
\end{align}
\end{lemma}
\begin{proof}
Substituting $\Bd_h$ in \eqref{proj_mix_poisson_1} with $\breve{\BJ}_h-\Cr_d\BJ$ gives
\begin{align}
\|\breve{\BJ}_h-\Cr_d\BJ\|_{\BL^2(\DOmega_h)}^2 &=(\BJ-\Cr_d\BJ,\breve{\BJ}_h-\Cr_d\BJ)\nn\\
&\leq \|\BJ-\Cr_d\BJ\|_{\BL^2(\DOmega_h)}\|\breve{\BJ}_h-\Cr_d\BJ\|_{\BL^2(\DOmega_h)},
\end{align}
which means $\|\Cr_d\BJ-\breve{\BJ}_h\|_{\BL^2(\DOmega_h)} \leq \|\BJ-\Cr_d\BJ\|_{\BL^2(\DOmega_h)}$.
Using triangle inequality and Lemma \ref{lem:proj_J}, we obtain
\begin{align}\label{lem:poisson_proj_J}
\|\BJ-\breve{\BJ}_h\|_{\BL^2(\DOmega_h)}
\leq &\|\BJ-\Cr_d\BJ\|_{\BL^2(\DOmega_h)}
+ \|\Cr_d\BJ-\breve{\BJ}_h\|_{\BL^2(\DOmega_h)}\nn\\
\leq &2\|\BJ-\Cr_d\BJ\|_{\BL^2(\DOmega_h)}
\lesssim h^k\|\BJ\|_{\BH^k(D_H)}.
\end{align}
Using the inf-sup condition, for any $\Bd_h\in\BD_h^{k-1}(\Div 0)$, we have
\begin{align*}
\|\breve{\phi}_h-\Cp_\varphi\phi\|_{L^2(\DOmega_h)}
\leq &\frac{(\breve{\phi}_h-\Cp_\varphi\phi,\Div\Bd_h)_h}{\|\Bd_h\|_{\BH(\Div,\DOmega_h)}}\nn\\
\lesssim &\frac{(\BJ-\breve{\BJ}_h,\Bd_h)_h-(\phi-\Cp_\varphi\phi,\Div\Bd_h)_h}{\|\Bd\|_{\BH(\Div,\DOmega_h)}}\nn\\
\lesssim &\|\BJ-\breve{\BJ}_h\|_{\BL^2(\DOmega_h)}+\|\phi-\Cp_\varphi\phi\|_{L^2(\DOmega_h)}.
\end{align*}
Consequently, using the result of Lemma \ref{lem:proj_phi} gives
\begin{align}\label{lem:poisson_proj_phi}
\|\phi-\breve{\phi}_h\|_{L^2(\DOmega_h)}
\leq&\|\phi-\Cp_\varphi\phi\|_{L^2(\DOmega_h)}+\|\BJ-\breve{\BJ}_h\|_{\BL^2(\DOmega_h)}\nn\\
\lesssim& h^{k-1}(\|\BJ\|_{\BH^k(D_H)}+\|\phi\|_{H^{k-1}(D_H)}).
\end{align}
Combining \eqref{lem:poisson_proj_J} and \eqref{lem:poisson_proj_phi} yields the estimate.
\qed
\end{proof}
\vspace{5mm}

Based on the above lemmas, we will present the main results of this paper.
\begin{theorem}\label{the:main_error}
Assume that the extended exact solution $(\Bu,p,\BJ,\phi)\in\BH^{k+1}(D_H)\times H^k(D_H)\times\BH^k(D_H)\times H^{k-1}(D_H)$ and that $\BB\in\BL^{\infty}(D_H)$,
the numerical solution $(\Bu_h,p_h,\BJ_h,\phi_h)$ of \eqref{eq:model_8st-0} satisfies
\begin{align*}
&|||(\Bu-\Bu_h,\BJ-\BJ_h)|||_{*}+\|p-p_h\|_{L^2(\DOmega_h)}+h\|\phi-\phi_h\|_{L^2(\DOmega_h)}\nn\\
\leq& {C(R_e,\alpha,\gamma,\|\BB\|_{\BL^\infty(D_H)})}h^k\big(\|\Bu\|_{\BH^{k+1}(D_H)}+\|\BJ\|_{\BH^{k}(D_H)}\nn\\
&+\|p\|_{H^k(D_H)}+\|\phi\|_{H^{k-1}(D_H)}+|||(\Bf,\Bg)|||_{1}\big),
\end{align*}
where $$|||(\Bu,\BJ)|||_{*}^2:=|||\Bu|||_{\Cd}^2+\alpha\|\BJ\|_{\BL^2(\DOmega_h)}^2$$
Note that the finite element order for $\BJ_h$
is $k-1$, thus the $\BL^2$ error of current density $\BJ$ is optimal.
\end{theorem}
\begin{proof}
By restricting \eqref{eq:model_2st} on $\DOmega_h$ , we have
\begin{subequations}\label{eq:err_1st}
\begin{align}
&\Ca(\Bu,\BJ;\Bv_h,\Bd_h)_h + \gamma(\Div\Bu,\Div\Bv_h)_h+ \Co(\Bu;\Bu,\Bv_h)_h
- \Cb(p,\phi;\Bv_h,\Bd_h)_h\nn\\
&+ \alpha\langle\phi,\Bn\cdot\Bd_h\rangle_{\partial\DOmega_h} = \Cl(\Bf,\Bg,\Bv_h,\Bd_h)_h,\\
&\Cb(q_h,\varphi_h;\Bu,\BJ)_h = 0.
\end{align}
\end{subequations}
for any $(\Bv_h,q_h,\Bd_h,\varphi_h) \in \BV_{h,0}^k\times Q_h^{k-1}\times\BD_h^{k-1}\times S_h^{k-2}$.
Note that $\phi$ is nonzero on $\partial\DOmega_h$, even though it vanishes on $\partial\DOmega$.

Subtracting \eqref{eq:model_8st} from \eqref{eq:err_1st}, we have the error equation
\begin{subequations}\label{eq:err_2st}
\begin{align}
&\Ca(\Bu-\Bu_h,\BJ-\BJ_h;\Bv_h,\Bd_h)_h + \gamma(\Div(\Bu-\Bu_h),\Div\Bv_h)_h + \Co(\Bu;\Bu,\Bv_h)_h\nn\\
&-\Co(\Bu_h;\Bu_h,\Bv_h)_h - \Cb(p-p_h,\phi-\phi_h;\Bv_h,\Bd_h)_h + \alpha\langle\phi,\Bn\cdot\Bd_h\rangle_{\partial\DOmega_h}\nn\\
&= \Cl(\Bf-\Bf_h,\Bg-\Bg_h,\Bv_h,\Bd_h)_h,\\
&\Cb(q_h,\varphi_h;\Bu-\Bu_h,\BJ-\BJ_h)_h = 0.
\end{align}
\end{subequations}
Setting
\begin{align*}
&\Bu-\Bu_h = \Bu-\breve{\Bu}_h + (\breve{\Bu}_h - \Bu_h) = \Be_\Bu + \eta_\Bu,\\
&p - p_h = p - \breve{p}_h + (\breve{p}_h - p_h) = e_p + \eta_p,\\
&\BJ-\BJ_h = \BJ-\breve{\BJ}_h + (\breve{\BJ}_h - \BJ_h) = \Be_\BJ + \eta_\BJ,\\
&\phi - \phi_h = \phi - \breve{\phi}_h + (\breve{\phi}_h - \phi_h) = e_\phi + \eta_\phi.
\end{align*}
With the Stokes projection \eqref{eq:stoke_proj} and the mixed Poisson projection \eqref{proj_mix_poisson},
setting $(\Bv_h,q_h,\Bd_h,\varphi_h)=(\eta_\Bu,\eta_p,\eta_\BJ,\eta_\phi)$, \eqref{eq:err_2st} can be rewritten as
\begin{align}\label{eq:err_3st}
& R_e^{-1}(\nabla\eta_\Bu,\nabla\eta_\Bu)_h+\gamma(\Div\eta_\Bu,\Div\eta_\Bu)_h + \alpha(\eta_\BJ,\eta_\BJ)_h\nn\\
= &\Cl(\Bf-\Bf_h,\Bg-\Bg_h,\eta_\Bu,\eta_\BJ)_h + \Co(\Bu_h;\Bu_h,\eta_\Bu)_h - \Co(\Bu;\Bu,\eta_\Bu)_h\nn\\
&+ \alpha\langle\phi,\Bn\cdot\eta_\BJ\rangle_{\partial\DOmega_h}
+ \alpha(\Be_\BJ\times\BB,\eta_\Bu)_h + \alpha(\Be_\Bu\times\BB,\eta_\BJ)_h.
\end{align}

Using the property \eqref{pro_Phi} and Young's inequality, the source term becomes
\begin{align}\label{eq:err_3st_r1}
&\Cl(\Bf-\Bf_h,\Bg-\Bg_h;\eta_\Bu,\eta_\BJ)_h\nn\\
\leq &\|\Bf-\Bf\circ\Phi_h^{-1}\|_{\BL^2(\DOmega_h)}\|\eta_\Bu\|_{\BL^2(\DOmega_h)} + \alpha\|\Bg-\Bg\circ\Phi_h^{-1}\|_{\BL^2(\DOmega_h)}\|\eta_\BJ\|_{\BL^2(\DOmega_h)}\nn\\
\leq &h^{k+1}\|\Bf\|_{\BW^{1,\infty}(\DOmega_h)}\|\nabla\eta_\Bu\|_{\BL^2(\DOmega_h)} + \alpha h^{k+1}\|\Bg\|_{\BW^{1,\infty}(\DOmega_h)}\|\eta_\BJ\|_{\BL^2(\DOmega_h)}\nn\\
\leq &C(R_e,\alpha,\varepsilon)h^{2k+2}|||(\Bf,\Bg)|||_{1}^2+\varepsilon R_e|||(\eta_\Bu,\eta_\BJ)|||_{*}^2
\end{align}
where $|||(\Bf,\Bg)|||_{1}^2:=\|\Bf\|_{\BW^{1,\infty}(D_H)}^2+\|\Bg\|_{\BW^{1,\infty}(D_H)}^2$.

Applying embedding $\BH_0^1(\DOmega_h)\hookrightarrow\BL^6(\DOmega_h)$ and the stability bound of $|||(\Bu,\BJ)|||_{\Ca}$, $|||(\Bu_h,\BJ_h)|||_{\Ca_h}$, we have
\begin{align}\label{eq:err_3st_r2}
&\Co(\Bu_h;\Bu_h,\eta_\Bu)_h-\Co(\Bu;\Bu,\eta_\Bu)_h\nn\\
= &\Co(\Bu_h-\Bu;\Bu,\eta_\Bu)_h+\Co(\Bu_h;\Bu_h-\Bu,\eta_\Bu)_h\nn\\
= &\Co(-\Be_\Bu;\Bu,\eta_\Bu)_h+\Co(-\eta_\Bu;\Bu,\eta_\Bu)_h+\Co(\Bu_h;-\Be_\Bu,\eta_\Bu)_h\nn\\
\leq &\frac{1}{2}\|\Be_\Bu\|_{\BL^6(\DOmega_h)}\|\nabla\Bu\|_{\BL^2(\DOmega_h)}\|\eta_\Bu\|_{\BL^3(\DOmega_h)}
+ \frac{1}{2}\|\Be_\Bu\|_{\BL^6(\DOmega_h)}\|\nabla\eta_\Bu\|_{\BL^2(\DOmega_h)}\|\Bu\|_{\BL^3(\DOmega_h)}\nn\\
+ &\frac{1}{2}\|\eta_\Bu\|_{\BL^6(\DOmega_h)}\|\nabla\Bu\|_{\BL^2(\DOmega_h)}\|\eta_\Bu\|_{\BL^3(\DOmega_h)}
+ \frac{1}{2}\|\eta_\Bu\|_{\BL^6(\DOmega_h)}\|\nabla\eta_\Bu\|_{\BL^2(\DOmega_h)}\|\Bu\|_{\BL^3(\DOmega_h)}\nn\\
+ &\frac{1}{2}\|\Bu_h\|_{\BL^6(\DOmega_h)}\|\nabla\Be_\Bu\|_{\BL^2(\DOmega_h)}\|\eta_\Bu\|_{\BL^3(\DOmega_h)}
+ \frac{1}{2}\|\Bu_h\|_{\BL^6(\DOmega_h)}\|\nabla\eta_\Bu\|_{\BL^2(\DOmega_h)}\|\Be_\Bu\|_{\BL^3(\DOmega_h)}\nn\\
\lesssim &\|\nabla\eta_\Bu\|_{\BL^2(\DOmega_h)}\|\nabla\Be_\Bu\|_{\BL^2(\DOmega_h)}
(\|\nabla\Bu\|_{\BL^2(\DOmega_h)}+\|\nabla\Bu_h\|_{\BL^2(\DOmega_h)})\nn\\
& +\|\nabla\eta_\Bu\|_{\BL^2(\DOmega_h)}^2\|\nabla\Bu\|_{\BL^2(\DOmega_h)}\nn\\
\lesssim &C(R_e,\alpha,\varepsilon)\|\nabla\Be_\Bu\|_{\BL^2(\DOmega_h)}^2
+ \varepsilon R_e|||(\eta_\Bu,\eta_\BJ)|||_{*}^2.
\end{align}
Applying the Green's formula and Cauchy-Schwarz inequality, we obtain
\begin{align}\label{eq:err_3st_r4}
&\alpha\langle \phi,\Bn\cdot\eta_\BJ\rangle_{\partial\DOmega_h}
= \alpha\langle \phi-\phi^\sharp,\Bn\cdot\eta_\BJ\rangle_{\partial\DOmega_h}\nn\\
= &\alpha(\nabla(\phi-\phi^\sharp),\eta_\BJ)_h+\alpha(\phi-\phi^\sharp,\Div\eta_\BJ)_h
\lesssim \alpha\|\phi-\phi^\sharp\|_{H^1(\DOmega_h)}\|\eta_\BJ\|_{\BL^2(\DOmega_h)}\nn\\
\lesssim& C(\alpha,\varepsilon) \|\phi-\phi^\sharp\|_{H^1(\DOmega_h)}^2+\varepsilon|||(\eta_\Bu,\eta_\BJ)|||_{*}^2.
\end{align}
Using the Young's inequality gives
{\begin{align}\label{eq:err_3st_r3}
&\alpha(\Be_\BJ\times\BB,\eta_\Bu)_h+\alpha(\Be_\Bu\times\BB,\eta_\BJ)_h\nn\\
\leq&\alpha\|\BB\|_{\BL^\infty(\DOmega_h)} (\|\Be_\BJ\|_{\BL^2(\DOmega_h)}\|\eta_\Bu\|_{\BL^2(\DOmega_h)}
+ \|\Be_\Bu\|_{\BL^2(\DOmega_h)}\|\eta_\BJ\|_{\BL^2(\DOmega_h)})\nn\\
\leq& C(R_e,\alpha,\varepsilon,\|\BB\|_{\BL^\infty(\DOmega_h)})(\|\Be_\BJ\|_{\BL^2(\DOmega_h)}^2\nn\\
&+\|\Be_\Bu\|_{\BL^2(\DOmega_h)}^2)
+ \varepsilon R_e|||(\eta_\Bu,\eta_\BJ)|||_{*}^2
\end{align}
Combing \eqref{eq:err_3st}-\eqref{eq:err_3st_r3} and Young's inequality yields
\begin{align}
&(1-3\varepsilon R_e-\varepsilon)|||(\eta_\Bu,\eta_\BJ)|||_{*}^2\nn\\
\lesssim &C(R_e,\alpha,\varepsilon)h^{2k+2}|||(\Bf,\Bg)|||_{1}^2
+ C(R_e,\alpha,\varepsilon)\|\nabla\Be_\Bu\|_{\BL^2(\DOmega_h)}^2+C(\alpha,\varepsilon) \|\phi-\phi^\sharp\|_{H^1(\DOmega_h)}^2\nn\\
&
+ C(R_e,\alpha,\varepsilon,\|\BB\|_{\BL^\infty(\DOmega_h)})(\|\Be_\BJ\|_{\BL^2(\DOmega_h)}^2 + \|\Be_\Bu\|_{\BL^2(\DOmega_h)}^2)
\end{align}
If $\varepsilon<\frac{1}{3R_e+1}$, it follows from Lemma \ref{lem:phi-phi_sharp} that
\begin{align}
&|||(\eta_\Bu,\eta_\BJ)|||_{*}\nn\\
\lesssim &C(R_e,\alpha,\|\BB\|_{\BL^\infty(\DOmega_h)})(h^{k+1}|||(\Bf,\Bg)|||_{1}+\|\nabla\Be_\Bu\|_{\BL^2(\DOmega_h)}
+ \|\phi-\phi^\sharp\|_{H^1(\DOmega_h)}\nn\\
&+ \|\Be_\BJ\|_{\BL^2(\DOmega_h)})\nn\\
\lesssim &C(R_e,\alpha,\|\BB\|_{\BL^\infty(D_H)})h^k(|||(\Bf,\Bg)|||_{1}+\|\Bu\|_{\BH^{k+1}(D_H)}+\|\BJ\|_{\BH^{k}(D_H)}\nn\\
&+ \|p\|_{H^{k}(D_H)}+\|\phi\|_{H^2(D_H)}).
\end{align}
Combining \eqref{eq:stoke_proj_prop}, \eqref{eq:poisson_proj_prop} and triangle inequality, we obtain
\begin{align}\label{eq_error_uJ}
&|||(\Bu-\Bu_h,\BJ-\BJ_h)|||_{*}\nn\\
&\lesssim C(R_e,\alpha,\gamma,\|\BB\|_{\BL^\infty(D_H)})h^k(|||(\Bf,\Bg)|||_{1}+\|\Bu\|_{\BH^{k+1}(D_H)}+\|\BJ\|_{\BH^{k}(D_H)}\nn\\
&+\|p\|_{H^{k}(D_H)}+\|\phi\|_{H^2(D_H)}).
\end{align}

By the inf-sup condition \eqref{infsup_dis_up}, using Lemma \ref{lem:proj_p} and \eqref{eq_error_uJ} gives
\begin{align}
&\|\eta_p\|_{L^2(\DOmega_h)}\nn\\
\lesssim &\sup_{\Bv_h\in\BV_{h,0}^k}\frac{(\eta_p,\Div\Bv_h)_h}{\|\Bv_h\|_{\BH^1(\DOmega_h)}}
\lesssim \sup_{\Bv_h\in\BV_{h,0}^k}\frac{(-e_p+p-p_h,\Div\Bv_h)_h}{\|\Bv_h\|_{\BH^1(\DOmega_h)}}\nn\\
\lesssim &\|e_p\|_{L^2(\DOmega_h)}
+ \frac{\Co(\Bu-\Bu_h;\Bu,\Bv_h)_h+ \Co(\Bu_h;\Bu-\Bu_h,\Bv_h)_h}{\|\Bv_h\|_{\BH^1(\DOmega_h)}}\nn\\
&+ \frac{ R_e^{-1}(\nabla(\Bu-\Bu_h),\nabla\Bv_h)_h+\gamma(\Div(\Bu-\Bu_h),\Div\Bv_h)_h}{\|\Bv_h\|_{\BH^1(\DOmega_h)}}\nn\\
&+ \frac{\alpha((\BJ-\BJ_h)\times\BB,\Bv_h)_h+(\Bf-\Bf_h,\Bv_h)_h}{\|\Bv_h\|_{\BH^1(\DOmega_h)}}\nn\\
\lesssim &\|e_p\|_{L^2(\DOmega_h)}+|||(\Bu-\Bu_h,\BJ-\BJ_h)|||_{*}+h^{k+1}\|\Bf\|_{\BW^{1,\infty}(\DOmega_h)}\nn\\
\lesssim & C(R_e,\alpha,\gamma,\|\BB\|_{\BL^\infty(D_H)})h^k\big(|||(\Bf,\Bg)|||_{1}+\|\Bu\|_{\BH^{k+1}(D_H)}+\|\BJ\|_{\BH^{k}(D_H)}\nn\\
& +\|p\|_{H^k(D_H)}+\|\phi\|_{H^2(D_H)}\big).
\end{align}
Similarly, based on \eqref{infsup_dis_Jphi}, \eqref{eq:lem_proj_phi} and \eqref{eq_error_uJ}, we obtain
\begin{align}
\|\eta_\phi\|_{L^2(\DOmega_h)}
\lesssim &\sup_{\Bd_h\in\BD_h^{k-1}}\frac{(-e_\phi+\phi-\phi_h,\Div\Bd_h)_h}{\|\Bd_h\|_{\BH(\Div,\DOmega_h)}}\nn\\
\lesssim &\|e_\phi\|_{L^2(\DOmega_h)}+\frac{(\BJ-\BJ_h,\Bd_h)_h+\langle \phi,\Bn\cdot\Bd_h\rangle_{\partial\DOmega_h}}{\|\Bd_h\|_{\BH(\Div,\DOmega_h)}}\nn\\
& +\frac{((\Bu-\Bu_h)\times\BB,\Bd_h)_h+(\Bg-\Bg_h,\Bd_h)_h}{\|\Bd_h\|_{\BH(\Div,\DOmega_h)}}\nn\\
\lesssim &\|e_\phi\|_{L^2(\DOmega_h)}+\|\BJ-\BJ_h\|_{\BL^2(\DOmega_h)}+\|\phi-\phi^\sharp\|_{H^1(\DOmega_h)}\nn\\
& +\|\Bu-\Bu_h\|_{\BH^1(\DOmega_h)}+h^{k+1}\|\Bg\|_{\BW^{1,\infty}(\DOmega_h)}\nn\\
\lesssim &C(R_e,\alpha,\gamma,\|\BB\|_{\BL^\infty(D_H)})h^{k-1}(|||(\Bf,\Bg)|||_{1}+\|\Bu\|_{\BH^{k+1}(D_H)}\nn\\
&+\|\BJ\|_{\BH^{k}(D_H)} +\|p\|_{H^k(D_H)}+\|\phi\|_{H^{k-1}(D_H)}).
\end{align}
}
Combing the triangle inequality completes the proof.
\qed
\end{proof}

\subsection{Optimal error estimate in the $\BL^2$-norm for velocity}
In Theorem \ref{the:main_error}, the optimal $\BL^2$ error estimate for $\BJ$ has been proven. In this subsection, we mainly prove the optimal $\BL^2$-norm error estimate for $\Bu$, which is precisely stated in the following theorem.
\begin{theorem}\label{the:main_error1}
Under the assumptions of Theorem \ref{the:main_error}, we have
\begin{align*}
\|\Bu-\Bu_h\|_{\BL^2(\DOmega_h)}
\leq& {C(R_e,\alpha,\gamma,\|\BB\|_{\BL^\infty(D_H)})}  h^{k+1}\big(|||(\Bf,\Bg)|||_{1}
+ \|\Bu\|_{\BH^{k+1}(D_H)}\nn\\
&+\|\BJ\|_{\BH^{k}(D_H)} + \|p\|_{H^k(D_H)}+\|\phi\|_{H^{k-1}(D_H)}\big).
\end{align*}
\end{theorem}
\begin{proof}
For the extended true solutions, we have
\begin{align}\label{eq:model_9st}
&\Co(\Bu;\Bu,\Bv_h)_h + R_e^{-1}(\nabla\Bu,\nabla\Bv_h)_h + \gamma(\Div\Bu,\Div\Bv_h)_h\nn\\
&\quad - (p,\Div\Bv_h)_h- \alpha(\BJ\times\BB,\Bv_h)_h= (\Bf,\Bv_h)_h,
\end{align}
Subtracting \eqref{eq:model_8st-0:1} from \eqref{eq:model_9st}, we have the error equation
\begin{align}\label{eq_err_l2}
&\Co(\Bu;\Bu,\Bv_h)_h- \Co(\Bu_h;\Bu_h,\Bv_h)_h+ R_e^{-1}(\nabla(\Bu-\Bu_h),\nabla\Bv_h)_h\nn\\
&+ \gamma(\Div(\Bu-\Bu_h),\Div\Bv_h)_h- (p-p_h,\Div\Bv_h)_h- \alpha((\BJ-\BJ_h)\times\BB,\Bv_h)_h\nn\\
= &(\Bf-\Bf_h,\Bv_h)_h.
\end{align}
To simplify the writing, we set
\begin{align*}
&\Bu-\Bu_h = \Bu-\breve{\Bu}_h + (\breve{\Bu}_h - \Bu_h) = \Be_\Bu + \eta_\Bu,\\
&p - p_h = p - \breve{p}_h + (\breve{p}_h - p_h) = e_p + \eta_p.
\end{align*}
Take $(\Bv_h,q_h)= (\eta_\Bu,\eta_p)$, and substitute \eqref{eq:stoke_proj} into \eqref{eq_err_l2} to obtain
\begin{align}\label{eq:nabla_eta_u}
& R_e^{-1}(\nabla\eta_\Bu,\nabla\eta_\Bu)_h+\gamma(\Div\eta_\Bu,\Div\eta_\Bu)_h\nn\\
&\hspace{-0.6cm}=(\Bf-\Bf_h,\eta_\Bu)_h-\Co(\Bu;\Bu,\eta_\Bu)_h- \Co(\Bu_h;\Bu_h,\eta_\Bu)_h\nn\\
&\hspace{-0.6cm}\quad + \alpha((\BJ-\BJ_h)\times\BB,\eta_\Bu)_h.
\end{align}
Using the Young's inequality to the first term on the right-hand side leads to
\begin{align}\label{eq:l2error_right_f}
(\Bf-\Bf_h,\eta_\Bu)_h
= &\|\Bf-\Bf\circ\Phi_h^{-1}\|_{\BL^2(\DOmega_h)}\|\eta_\Bu\|_{\BL^2(\DOmega_h)} \nn\\
\leq &h^{k+1}\|\Bf\|_{\BW^{1,\infty}(\DOmega_h)}\|\nabla\eta_\Bu\|_{\BL^2(\DOmega_h)}\nn\\
\leq &\varepsilon^{-1}h^{2k+2}\|\Bf\|_{\BW^{1,\infty}(\DOmega_h)}^2
+ \varepsilon R_e|||\eta_\Bu|||_{\Cd}^2.
\end{align}
By applying \eqref{eq:tri_define} and embedding inequality, the trilinear term becomes
\begin{align}\label{eq:l2error_right_trilinear}
&\Co(\Bu_h;\Bu_h,\eta_\Bu)_h-\Co(\Bu;\Bu,\eta_\Bu)_h\nn\\
\leq &\frac{1}{2}\|\Be_\Bu\|_{\BL^2(\DOmega_h)}\|\nabla\Bu\|_{\BL^3(\DOmega_h)}\|\eta_\Bu\|_{\BL^6(\DOmega_h)}
+ \frac{1}{2}\|\Be_\Bu\|_{\BL^2(\DOmega_h)}\|\nabla\eta_\Bu\|_{\BL^2(\DOmega_h)}\|\Bu\|_{\BL^\infty(\DOmega_h)}\nn\\
& +\frac{1}{2}\|\eta_\Bu\|_{\BL^4(\DOmega_h)}\|\nabla\Bu\|_{\BL^2(\DOmega_h)}\|\eta_\Bu\|_{\BL^4(\DOmega_h)}
+ \frac{1}{2}\|\eta_\Bu\|_{\BL^4(\DOmega_h)}\|\nabla\eta_\Bu\|_{\BL^2(\DOmega_h)}\|\Bu\|_{\BL^4(\DOmega_h)}\nn\\
& +\|\Bu_h\|_{\BL^\infty(\DOmega_h)}\|\Be_\Bu\|_{\BL^2(\DOmega_h)}\|\nabla\eta_\Bu\|_{\BL^2(\DOmega_h)}
+ \frac{1}{2}\|\nabla\cdot\Bu_h\|_{\BL^3(\DOmega_h)}\|\Be_\Bu\|_{\BL^2(\DOmega_h)}\|\eta_\Bu\|_{\BL^6(\DOmega_h)}\nn\\
\lesssim &\big(\|\Be_\Bu\|_{\BL^2(\DOmega_h)}\|\nabla\Bu\|_{\BL^3(\DOmega_h)}\|\nabla\eta_\Bu\|_{\BL^2(\DOmega_h)}
+ \|\Be_\Bu\|_{\BL^2(\DOmega_h)}\|\nabla\eta_\Bu\|_{\BL^2(\DOmega_h)}\|\Bu\|_{\BL^\infty(\DOmega_h)}\nn\\
& +\|\nabla\eta_\Bu\|_{\BL^2(\DOmega_h)}^2\|\nabla\Bu\|_{\BL^2(\DOmega_h)}
+ \|\Bu_h\|_{\BL^\infty(\DOmega_h)}\|\Be_\Bu\|_{\BL^2(\DOmega_h)}\|\nabla\eta_\Bu\|_{\BL^2(\DOmega_h)}\nn\\
& +\|\nabla\Bu_h\|_{\BL^3(\DOmega_h)}\|\Be_\Bu\|_{\BL^2(\DOmega_h)}\|\nabla\eta_\Bu\|_{\BL^2(\DOmega_h)}\big)\nn\\
\lesssim &C(R_e,\alpha,\varepsilon)\|\Be_\Bu\|_{\BL^2(\DOmega_h)}^2+\varepsilon R_e|||\eta_\Bu|||_{\Cd}^2.
\end{align}

In order to obtain a more precise estimate, we introduce an auxiliary dual problem. Given $\chi\in\BH^1(\DOmega_h)$, find $(\Bj,\varphi)\in\BH^1(\DOmega_h)\times H^2(\DOmega_h)$ satisfying
\begin{subequations}\label{eq:newnew}
\begin{align}
\Bj + \nabla \varphi = \chi,\;\;&\textrm{in}\;\;\DOmega_h,\label{eq:auxi_j_1}\\
\Div\Bj = 0,\;\;&\textrm{in}\;\;\DOmega_h,\label{eq:auxi_j_2}\\
\varphi = 0,\;\;&\textrm{on}\;\;\partial\DOmega_h.
\end{align}
\end{subequations}
Taking the divergence of \eqref{eq:auxi_j_1} and using \eqref{eq:auxi_j_2}, we obtain
\begin{align*}
\Delta\varphi = \Div\chi,
\end{align*}
which implies the regularity
\begin{align*}
\|\Bj\|_{\BH^1(\DOmega_h)} \leq \|\nabla\varphi\|_{\BH^1(\DOmega_h)} + \|\chi\|_{\BH^1(\DOmega_h)}
\leq C\|\chi\|_{\BH^1(\DOmega_h)}.
\end{align*}
Let $\Bj_h$ be the finite element solution of $\Bj$ of \eqref{eq:newnew}  in the space $\BD_h^{k-1}(\Div 0)$  ($k \geq 2$).
Because $(\BJ-\BJ_h, \Bd_h)_h = 0$ for any $\Bd_h \in \BD_h^{k-1}(\Div 0)$.
By taking $\BL^2$-inner product of \eqref{eq:auxi_j_1} with $\BJ-\BJ_h$ gives
\begin{align}
&(\BJ-\BJ_h,\chi)_h = (\BJ-\BJ_h,\Bj)_h = (\BJ-\BJ_h,\Bj-\Bj_h)_h\nn\\
&\leq \|\BJ-\BJ_h\|_{\BL^2(\DOmega_h)}\|\Bj-\Bj_h\|_{\BL^2(\DOmega_h)} \lesssim h  \|\chi\|_{\BH^1(\DOmega_h)} \|\BJ-\BJ_h\|_{\BL^2(\Omega_h)}.
\end{align}
According to the Theorem \ref{the:main_error} and the definition of $\BH^{-1}$-norm,
\begin{align}
&\|\BJ-\BJ_h\|_{\BH^{-1}(\DOmega_h)}\nn\\
= &\sup_{\chi\in \BH(\Div,\DOmega_h)}\frac{(\BJ-\BJ_h,\chi)_h}{\|\chi\|_{\BH^1(\DOmega_h)}}
\lesssim h\|\BJ-\BJ_h\|_{\BL^2(\DOmega_h)}\nn\\
\leq & {C(R_e,\alpha,\gamma,\|\BB\|_{\BL^\infty(D_H)})}h^{k+1}\big(|||(\Bf,\Bg)|||_{1}+\|\Bu\|_{\BH^{k+1}(D_H)}+\|\BJ\|_{\BH^{k}(D_H)}\nn\\
& +\|p\|_{H^k(D_H)}+\|\phi\|_{H^2(D_H)}\big).
\end{align}
Using the H\"{o}lder's inequality and Young's inequality has
{\begin{align}\label{eq:l2error_right_times}
&((\BJ-\BJ_h)\times\BB,\eta_\Bu)_h\nn\\
\leq &\|\BJ-\BJ_h\|_{\BH^{-1}(\DOmega_h)}\|\BB\times\eta_\Bu\|_{\BH^1(\DOmega_h)}\nn\\
\lesssim &\|\BB\|_{\BL^\infty(\DOmega_h)}\|\BJ-\BJ_h\|_{\BH^{-1}(\DOmega_h)}\|\nabla\eta_\Bu\|_{\BL^2(\DOmega_h)}\nn\\
\lesssim &\varepsilon R_e|||\eta_\Bu|||_{\Cd}^2+\varepsilon^{-1}\|\BB\|_{\BL^\infty(\DOmega_h)}^2\|\BJ-\BJ_h\|_{\BH^{-1}(\DOmega_h)}^2.
\end{align}
Based on \eqref{eq:l2error_right_f}, \eqref{eq:l2error_right_trilinear} and \eqref{eq:l2error_right_times}, equation \eqref{eq:nabla_eta_u} can be rewritten as
\begin{align}
(1-3\varepsilon R_e)|||\eta_\Bu|||_{\Cd}^2\lesssim C(R_e,\alpha,\|\BB\|_{\BL^\infty(\DOmega_h)},\varepsilon)\Big(h^{2k+2}\|\Bf\|_{\BW^{1,\infty}(\DOmega_h)}^2\nn\\
+ \|\Be_\Bu\|_{\BL^2(\DOmega_h)}^2
 +\|\BJ-\BJ_h\|_{\BH^{-1}(\DOmega_h)}^2\Big).
\end{align}
If $\varepsilon<\frac{1}{3R_e}$, it is easy to obtain
\begin{align}
|||\eta_\Bu|||_{\Cd}
\leq &C(R_e,\alpha,\gamma,\|\BB\|_{\BL^\infty(D_H)})h^{k+1}\big(|||(\Bf,\Bg)|||_{1}+\|\Bu\|_{\BH^{k+1}(D_H)}\nn\\
&+\|\BJ\|_{\BH^{k}(D_H)} +\|p\|_{H^k(D_H)}+\|\phi\|_{H^{k-1}(D_H)}\big).
\end{align}
Using Poincar\'{e}'s inequality and the triangle inequality, we have
\begin{align}
\|\Bu-\Bu_h\|_{\BL^2(\DOmega_h)}
&\leq\|\Be_u\|_{\BL^2(\DOmega_h)}+R_e|||\eta_\Bu|||_{\Cd}\nn\\
& \leq C(R_e,\alpha,\gamma,\|\BB\|_{\BL^\infty(D_H)})h^{k+1}\big(|||(\Bf,\Bg)|||_{1}+\|\Bu\|_{\BH^{k+1}(D_H)}\nn\\
&+\|\BJ\|_{\BH^{k}(D_H)} +\|p\|_{H^k(D_H)}+\|\phi\|_{H^{k-1}(D_H)}\big).
\end{align}
}
The proof is completed.
\end{proof}

\section{Numerical experiments}
\label{sec:exp}
{In this section, we present numerical examples to validate the theoretical analysis. All computations are carried out using the PHG library \cite{ZhangPHG}. The nonlinear system is solved by the Picard iteration described in Appendix, with the linear systems solved using the augmented Lagrangian block preconditioner from \cite{lnz19x}. }
\begin{example}
The computational domain is given by an unit ball
\begin{align*}
\DOmega=\{\Bx\in\mathbb{R}^3:|\Bx|<1\}.
\end{align*}
Furthermore, for simplicity, we set
\begin{align*}
\BB=(1,0,0);\;\;R_e=1;\;\;\gamma=0.5;\;\;\alpha=1.
\end{align*}
The exact solution is chosen to be sufficiently smooth and is prescribed as follows:
\begin{align*}
\Bu&=(x^2+y^2+z^2-1)(y-z, z-x, x-y)^\top;\;\;
p=x^2+y^2+z^2-3/5;\\
\BJ&=(\sin y, \cos z, -x)^\top;\;\;
\phi=x^2+y^2+z^2-1;
\end{align*}
The corresponding source term $\Bf$ and $\Bg$ are derived accordingly.
For the spatial discretization, we employ second-order Taylor-Hood elements for the velocity $\Bu$ and the pressure $p$, and first-order parametric BDM elements for the current density $\BJ$; that is, we set $k=2$ in our implementation.
Firstly in Table \ref{tab:mesh}, we give the five meshes information and 3D view of meshes $\Ct_1^{(2)}$ and $\Ct_2^{(2)}$ are given in Fig. \ref{fig_order2grid}.
\begin{table}[!htbp]
  \centering
  \caption{Curved meshes and DOFs used in the convergence tests ($k=2$).}
  \label{tab:mesh}
    \begin{tabular}{ccccccc}
    \hline
  Mesh & $h_{\max}$  & $h_{\min}$ & Elements & DOFs of $\Bu_h$ & DOFs of $p_h$  & DOFs of $\BJ_h$  \cr
    \hline
   $\Ct_1^{(2)}$   &1.1281   &1.1281    &48    &375    &27    &360     \\
   $\Ct_2^{(2)}$  &0.9292   &0.5640     &384    &2187    &125    &2592     \\
   $\Ct_3^{(2)}$  &0.5971   &0.2820    &3072    &14739    &729    &19584     \\
   $\Ct_4^{(2)}$  &0.3347   &0.1410    &24576    &107811    &4913    &152064     \\
   $\Ct_5^{(2)}$  &0.1768   &0.0705    &196608    &823875    &35937    &1198080     \\
\hline
\end{tabular}
\end{table}

\begin{figure}[!htbp]
  \centering
  \includegraphics[width=0.35\textwidth]{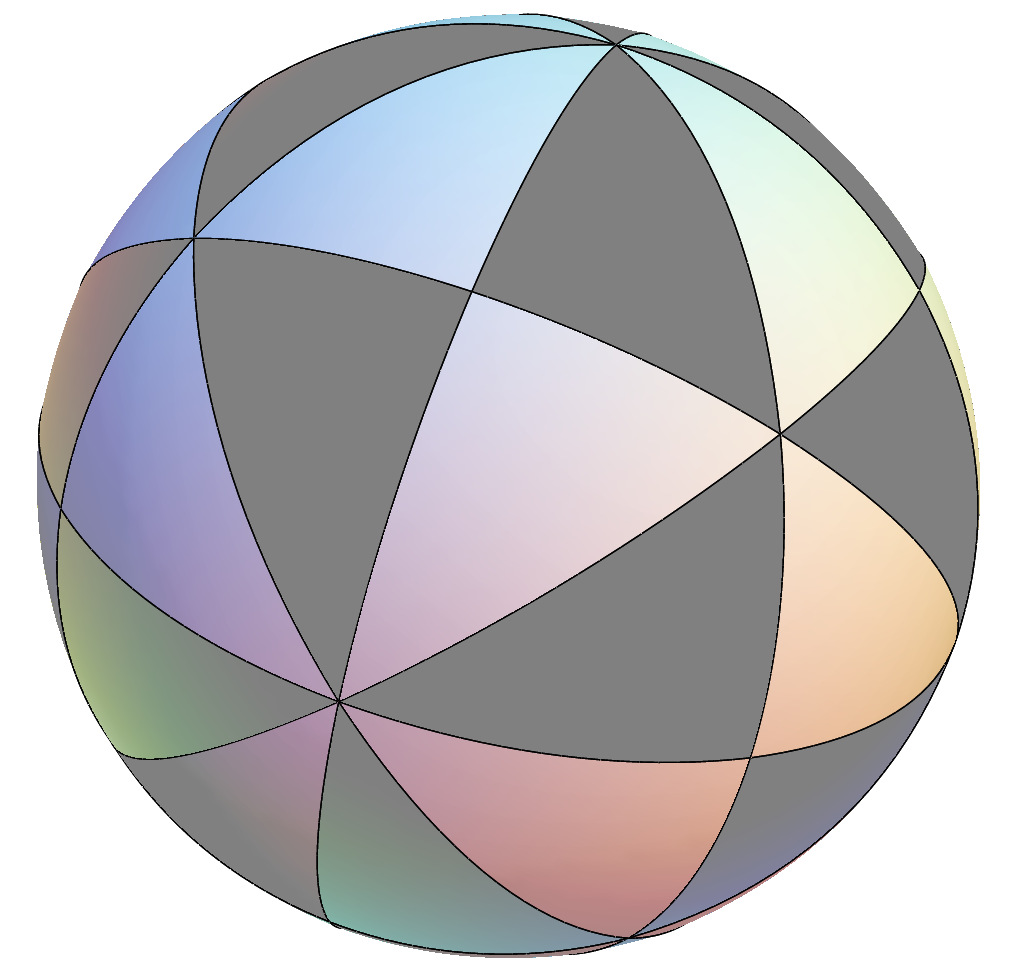}
  \includegraphics[width=0.36\textwidth]{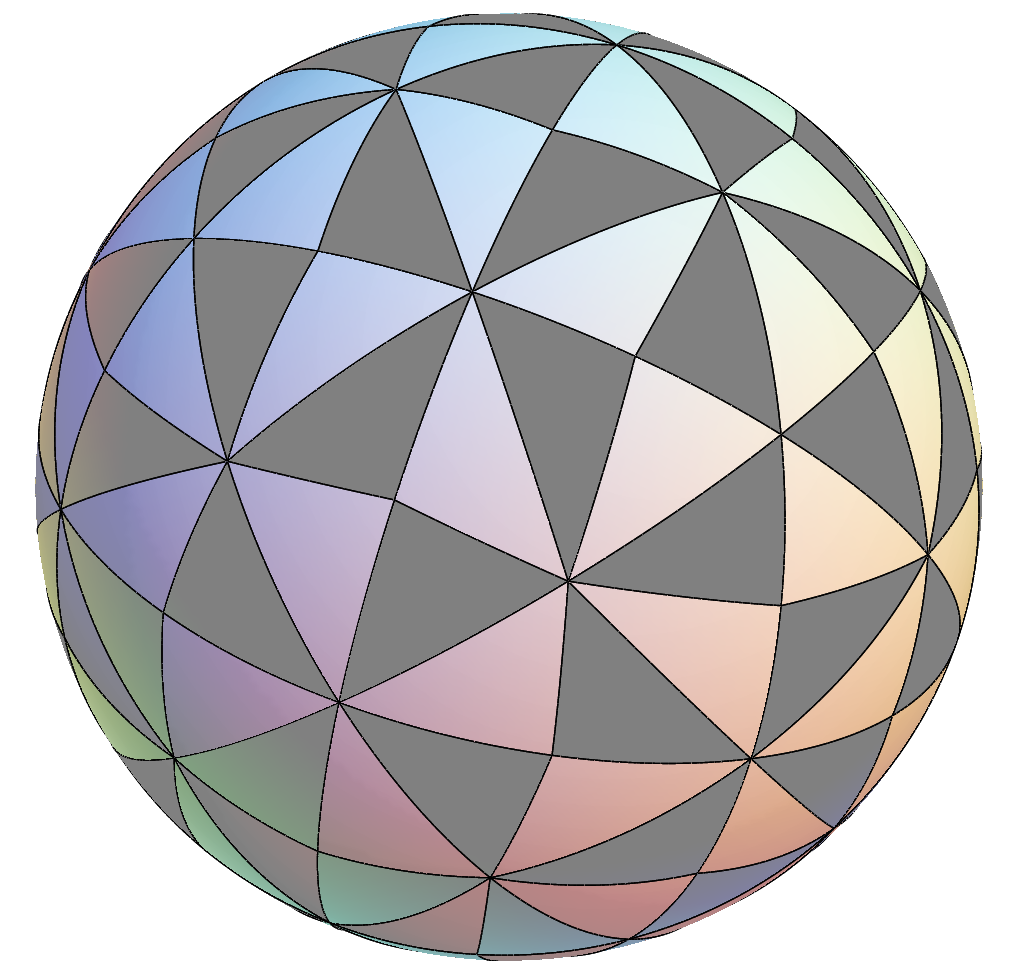}
  \caption{3D view of curved meshes $\Ct_1^{(2)}$ (left) and $\Ct_2^{(2)}$ (right).}\label{fig_order2grid}
\end{figure}

In Table \ref{tab:J-phi-straight}, we give the computed finite element errors of standard mixed finite element method \cite{lnz19}, namely the domain is approximated using
traditional straight meshes. Then in Table \ref{tab:J-phi-curve}, the finite element errors of the proposed parametric mixed finite element method in present work are shown.
We also plot the decreasing trend of the $\BL^2$-error for velocity  and current density in Fig. \ref{fig_err1}.
The $\BH^1$-error of $\Bu_h$ and $L^2$-error of $p_h$ are shown in Fig. \ref{fig_err2}. In 3D case, $h \sim N^{-1/3}$ and $N$ is the
number of degrees of freedoms of the corresponding physical field. From Fig. \ref{fig_err1} and Fig. \ref{fig_err2}, we know that for traditional method
\begin{equation}
\|\Bu-\hat{\Bu}_h\|_{\BL^2} \sim O(h^2),~~\|\BJ-\hat{\BJ}_h\|_{\BL^2} \sim O(h^{3/2}),~~\|\Div\hat{\BJ}_h\|_{L^2} \sim 0,
\end{equation}
\begin{equation}
\|\Bu-\hat{\Bu}_h\|_{\BH^1} \sim O(h^{3/2}),~~\|p-\hat{p}_h\|_{L^2} \sim O(h^{3/2}).
\end{equation}
Noting that $k=2$, namely although the divergence-free condition for $\BJ_h$ is satisfied,
the convergence rates are not optimal. While for our proposed parametric mixed finite element method, we have
\begin{equation*}
\|\Bu-\Bu_h\|_{\BL^2} \sim O(h^3),~~\|\BJ-\BJ_h\|_{\BL^2} \sim O(h^{2}),~~\|\Div\BJ_h\|_{L^2} \sim 0,
\end{equation*}
\begin{equation*}
\|\Bu-\Bu_h\|_{\BH^1} \sim O(h^{2}),~~\|p-p_h\|_{L^2} \sim O(h^{2}).
\end{equation*}
Namely the convergence rates are optimal and the discrete current density is also exactly divergence-free.
Those observations are in good agreement with Theorem \ref{the:main_error}, Theorem \ref{the:main_error1} and Lemma \ref{lem:div_J0}.
\begin{table}[!htbp]
  \centering
  \caption{Finite element errors of the standard mixed finite element method on straight meshes \cite{lnz19}.}
  \label{tab:J-phi-straight}
  {    \begin{tabular}{cccccc}
    \hline
 Mesh & $\|\Bu-\Bu_h\|_{\BL^2}$ & Order & $\|\BJ-\BJ_h\|_{\BL^2}$ & Order & $\|\Div\BJ_h\|_{L^2}$ \cr
    \hline
   $\widehat{\Ct}_1^{(2)}$  &3.2622e-01    &-     &2.4456e-01   &-       &1.5541e-13         \\
   $\widehat{\Ct}_2^{(2)}$  &1.0817e-01    &1.59  &9.5606e-02   &1.36    &3.2337e-13         \\
   $\widehat{\Ct}_3^{(2)}$  &2.7332e-02    &1.98  &2.9912e-02   &1.68    &1.2268e-13         \\
   $\widehat{\Ct}_4^{(2)}$  &6.6183e-03    &2.05  &9.4330e-03   &1.66    &3.2814e-13         \\
   $\widehat{\Ct}_5^{(2)}$  &1.6135e-03    &2.04  &3.1034e-03   &1.60    &7.4283e-13         \\
\hline
\end{tabular}
\\[3mm]
\begin{tabular}{ccccc}
    \hline
 Mesh & $\|\Bu-\Bu_h\|_{\BH^1}$   & Order & $\|p-p_h\|_{\BL^2}$  & Order \cr
    \hline
   $\widehat{\Ct}_1^{(2)}$   &4.4183e-01    &-     &2.0555e-01    &-       \\
   $\widehat{\Ct}_2^{(2)}$   &2.0448e-01    &1.11  &1.0348e-01    &0.99     \\
   $\widehat{\Ct}_3^{(2)}$   &7.8651e-02    &1.38  &4.0156e-02    &1.37     \\
   $\widehat{\Ct}_4^{(2)}$   &2.8828e-02    &1.45  &1.3601e-02    &1.56     \\
   $\widehat{\Ct}_5^{(2)}$   &1.0399e-02    &1.47  &4.2855e-03    &1.67     \\
\hline
\end{tabular}}
\end{table}

\begin{table}[!htbp]
  \centering
  \caption{Finite element errors of the proposed parametric mixed finite element method on curved meshes ($k=2$).}
  \label{tab:J-phi-curve}
  {  \begin{tabular}{cccccc}
    \hline
 Mesh & $\|\Bu-\Bu_h\|_{\BL^2}$ & Order & $\|\BJ-\BJ_h\|_{\BL^2}$ & Order & $\|\Div\BJ_h\|_{L^2}$ \cr
    \hline
   $\Ct_1^{(2)}$   &1.0476e-01    &-       &8.3757e-02  &-       &3.3235e-13         \\
   $\Ct_2^{(2)}$   &2.9791e-02    &1.81    &2.5082e-02  &1.74    &3.2982e-13         \\
   $\Ct_3^{(2)}$   &4.7740e-03    &2.64    &5.5435e-03  &2.18    &1.0841e-13         \\
   $\Ct_4^{(2)}$   &6.2814e-04    &2.93    &1.2129e-03  &2.19    &5.8018e-13         \\
   $\Ct_5^{(2)}$   &7.9277e-05    &2.99    &2.7963e-04  &2.12    &2.1900e-12         \\
\hline
\end{tabular}
\\[3mm]
\begin{tabular}{ccccc}
\hline
 Mesh & $\|\Bu-\Bu_h\|_{\BH^1}$   & Order & $\|p-p_h\|_{\BL^2}$  & Order \cr
    \hline
   $\Ct_1^{(2)}$  &2.7992e-01   &-     &1.0036e-01    &-         \\
   $\Ct_2^{(2)}$  &1.0047e-01   &1.48  &9.1572e-02    &0.13      \\
   $\Ct_3^{(2)}$  &2.9500e-02   &1.77  &2.3066e-02    &1.99      \\
   $\Ct_4^{(2)}$  &7.8687e-03   &1.91  &6.0426e-03    &1.93      \\
   $\Ct_5^{(2)}$  &2.0047e-03   &1.97  &1.6163e-03    &1.90      \\
\hline
\end{tabular}}
\end{table}

\begin{figure}[!htbp]
  \centering
  \includegraphics[width=0.95\textwidth]{uhJh.png}
  \caption{The decreasing trend of $\BL^2$-error of $\Bu_h$ and $\BL^2$-error of $\BJ_h$. The symbol with $\widehat{\Omega}_h$ indicates computation on straight mesh.}\label{fig_err1}
\end{figure}

\begin{figure}[!htbp]
  \centering
  \includegraphics[width=0.95\textwidth]{uhph.png}
  \caption{The decreasing trend of energy error of $\Bu_h$ and $L^2$-error of $p_h$. The symbol with $\widehat{\Omega}_h$ indicates computation on straight mesh. }\label{fig_err2}
\end{figure}

{
To verify the theoretical results for higher-order elements, we repeat the same
test with $k=3$, i.e., third-order Taylor-Hood elements for $\Bu$ and $p$,
and second-order parametric BDM elements for $\BJ$. Table~\ref{tab:dof_k3} lists the mesh information for this case, and the corresponding numerical errors are reported in Table~\ref{tab:J-phi-curve_k3}. The results demonstrate
\begin{equation*}
\|\Bu-\Bu_h\|_{\BL^2} \sim O(h^4),~~\|\BJ-\BJ_h\|_{\BL^2} \sim O(h^{3}),~~\|\Div\BJ_h\|_{L^2} \sim 0,
\end{equation*}
\begin{equation*}
\|\Bu-\Bu_h\|_{\BH^1} \sim O(h^{3}),~~\|p-p_h\|_{L^2} \sim O(h^{3}).
\end{equation*}
which are fully consistent with the optimal convergence rates in Theorem \ref{the:main_error}, Theorem \ref{the:main_error1} and Lemma \ref{lem:div_J0}.
}

\begin{table}[!htbp]
  \centering
{  \caption{Curved meshes and DOFs used in the convergence tests ($k=3$).}
  \label{tab:dof_k3}
    \begin{tabular}{ccccccc}
    \hline
  Mesh & $h_{\max}$  & $h_{\min}$ & Elements & DOFs of $\Bu_h$ & DOFs of $p_h$  & DOFs of $\BJ_h$  \cr
    \hline
   $\Ct_1^{(3)}$  &1.1281   &1.1281    &48     &1029     &125     &1008      \\
   $\Ct_2^{(3)}$  &0.9292   &0.5640    &384    &6591     &729     &7488      \\
   $\Ct_3^{(3)}$  &0.5971   &0.2820    &3072   &46875    &4913    &57600     \\
   $\Ct_4^{(3)}$  &0.3347   &0.1410    &24576  &352947   &35937   &451584     \\
\hline
\end{tabular}}
\end{table}

\begin{table}[!htbp]
  \centering
  {  \caption{Finite element errors of the proposed parametric mixed finite element method on curved meshes ($k=3$).}
   \label{tab:J-phi-curve_k3}
    \begin{tabular}{cccccc}
    \hline
Mesh & $\|\Bu-\Bu_h\|_{\BL^2}$ & Order & $\|\BJ-\BJ_h\|_{\BL^2}$ & Order & $\|\Div\BJ_h\|_{L^2}$ \cr
    \hline
   $\Ct_1^{(3)}$   &1.1065e-02 &-       &2.9152e-02  &-     &1.7329e-13         \\
   $\Ct_2^{(3)}$   &9.8165e-04 &3.49    &3.4303e-03  &3.09  &1.3160e-14         \\
   $\Ct_3^{(3)}$   &5.9388e-05 &4.05    &3.4396e-04  &3.32  &1.1532e-14         \\
   $\Ct_4^{(3)}$   &3.1021e-06 &4.26    &3.3182e-05  &3.37  &2.6662e-14         \\
\hline
\end{tabular}
\\[3mm]
\begin{tabular}{ccccccc}
\hline
 Mesh & $\|\Bu-\Bu_h\|_{\BH^1}$   & Order & $\|p-p_h\|_{\BL^2}$  & Order&  \cr
    \hline
   $\Ct_1^{(3)}$  &4.4909e-02 &-      &1.6900e-02  &-         \\
   $\Ct_2^{(3)}$  &7.1178e-03 &2.66   &1.4123e-02  &0.26       \\
   $\Ct_3^{(3)}$  &8.9639e-04 &2.99   &2.2112e-03  &2.68       \\
   $\Ct_4^{(3)}$  &9.3613e-05 &3.26   &2.4961e-04  &3.15       \\
\hline
\end{tabular}}
\end{table}
\end{example}

{\begin{example}
The computational domain is the spherical shell defined by
$$\DOmega=\{\Bx\in\mathbb{R}^3:0.4<|\Bx|<1\},$$
i.e., the unit ball with an open concentric ball of radius 0.4 removed.
The boundary conditions are
$$\Bu = 0,\;\;\textrm{on}\;\;|\Bx|=0.4\quad\text{and}\quad |\Bx|=1,$$
$$\phi = 0,\;\;\textrm{on}\;\;|\Bx|=0.4,\qquad\phi = 1,\;\;\textrm{on}\;\;|\Bx|=1.$$
The physical parameters are set to
$$\BB=(0,1,0),\;\; R_e=100,\;\; \gamma = 0.5,\;\; \alpha = 1.$$
The velocity field is discretised using isoparametric elements of degree $k=2$; other physical quantities are defined accordingly. The nonlinear system is solved using Picard iteration, and all relevant details are provided in the Appendix.

The imposed potential difference drives an electric current through the conducting fluid. The interaction of this current with the applied magnetic field produces a Lorentz force $\BJ\times\BB$, which drives fluid motion. The resulting velocity field, in turn, induces a current through the $\Bu\times\BB$ term in Ohm's law. Fig. \ref{fig:phy_dis} shows the distributions of the current density, velocity, electric potential, and pressure near the cross-section $z=0$. Fig. \ref{fig:velocity_dis} presents the velocity vector field near the plane $y=0$, together with the velocity magnitude on the planes $x=0$ and $y=0$. Since the magnetic field is oriented along the $y$-axis, the Lorentz force $\BJ\times\BB$ acts predominantly in directions perpendicular to $y$, driving fluid motion primarily in planes orthogonal to the magnetic field. Our numerical results actually capture this physical behaviour.

\begin{figure}[!htbp]
  \centering
{\includegraphics[width=0.48\textwidth]{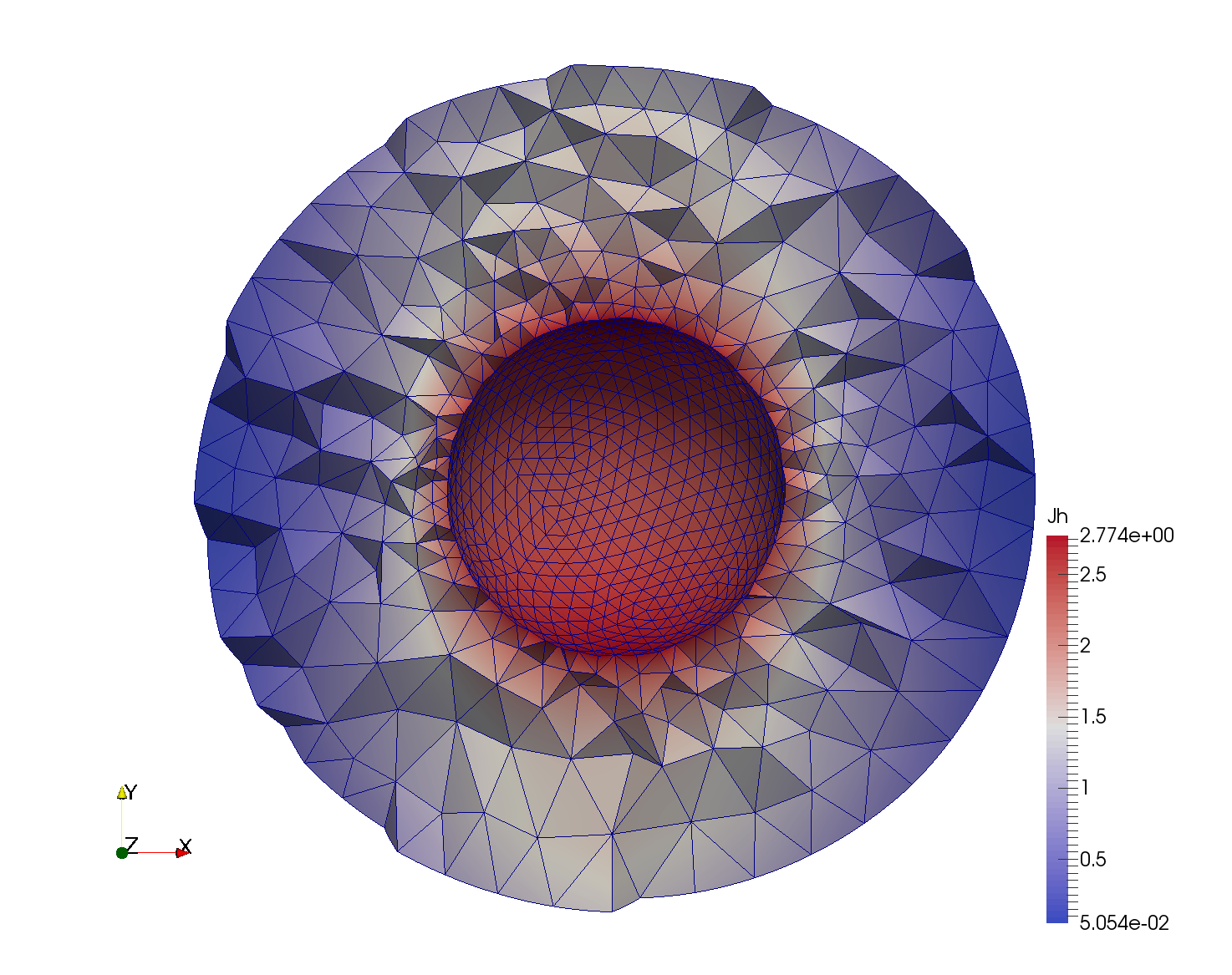}}
{\includegraphics[width=0.48\textwidth]{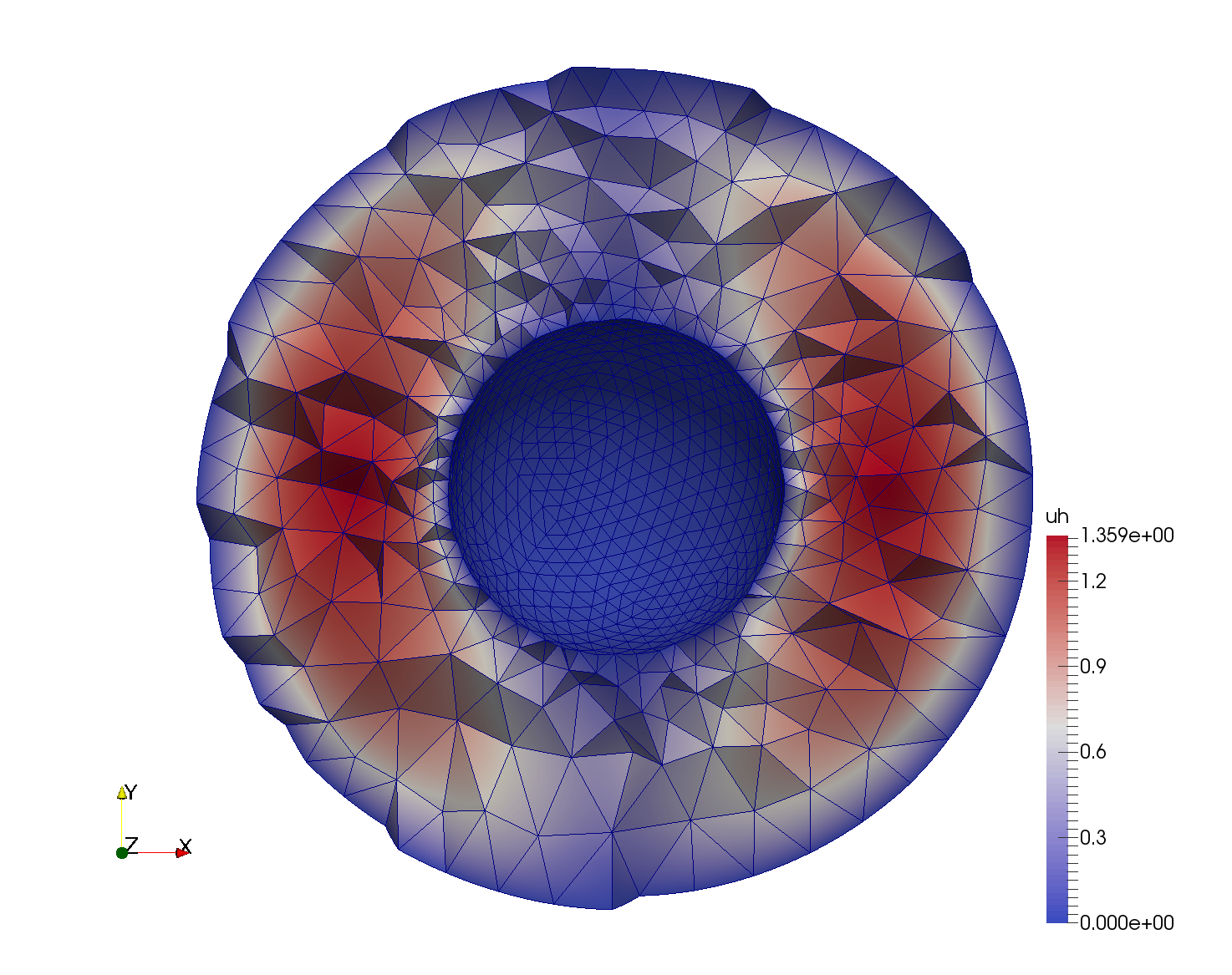}}\\
{\includegraphics[width=0.48\textwidth]{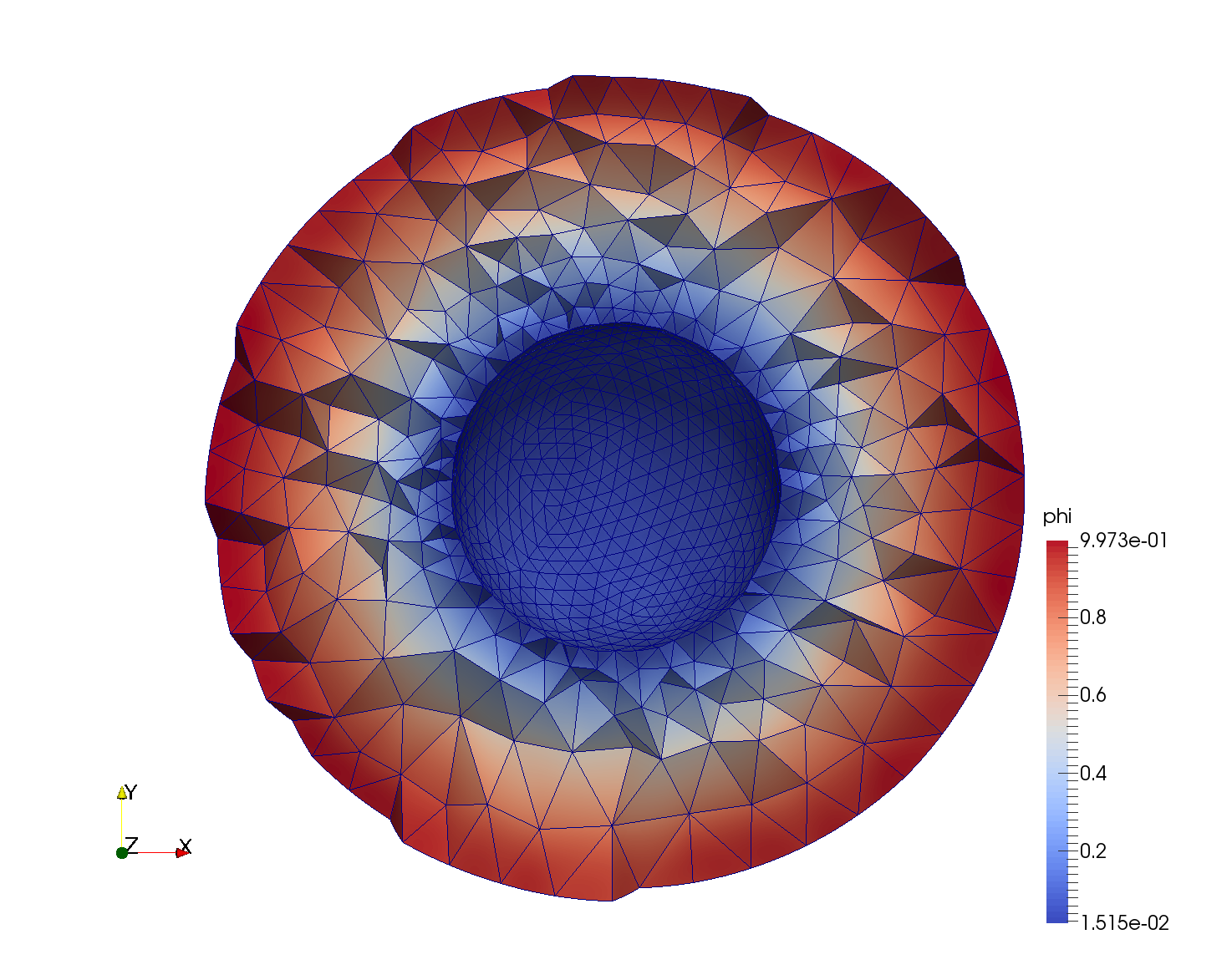}}
{\includegraphics[width=0.48\textwidth]{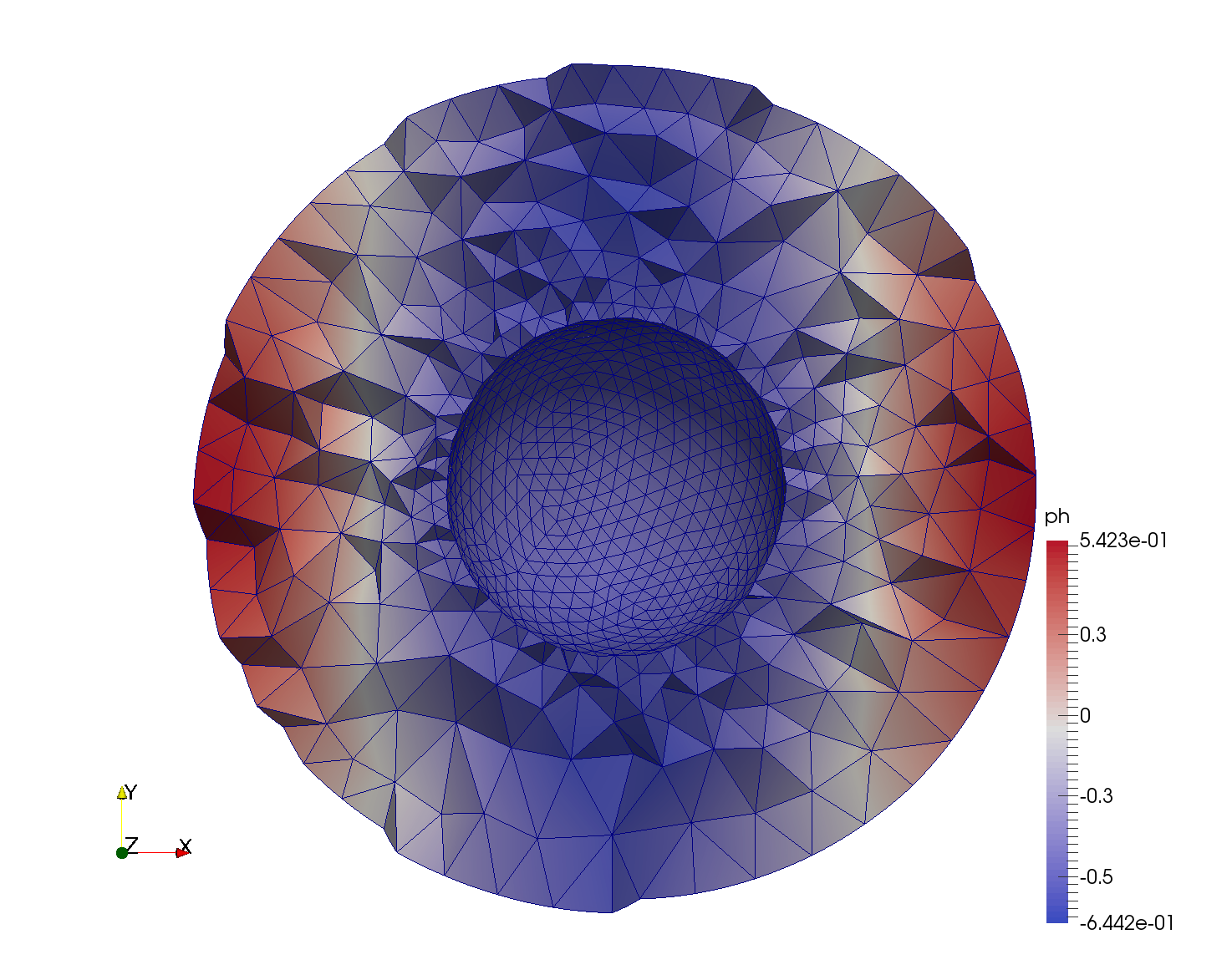}}
  \caption{Grid distribution and physical quantity distribution in the vicinity of cross-section $z=0$:
  Left top: current density $\BJ_h$; Right top velocity $\Bu_h$; Left down electric potential $\phi_h$;
  Right down: pressure $p_h$.}
  \label{fig:phy_dis}
\end{figure}

\begin{figure}[!htbp]
  \centering
  \includegraphics[width=0.48\textwidth]{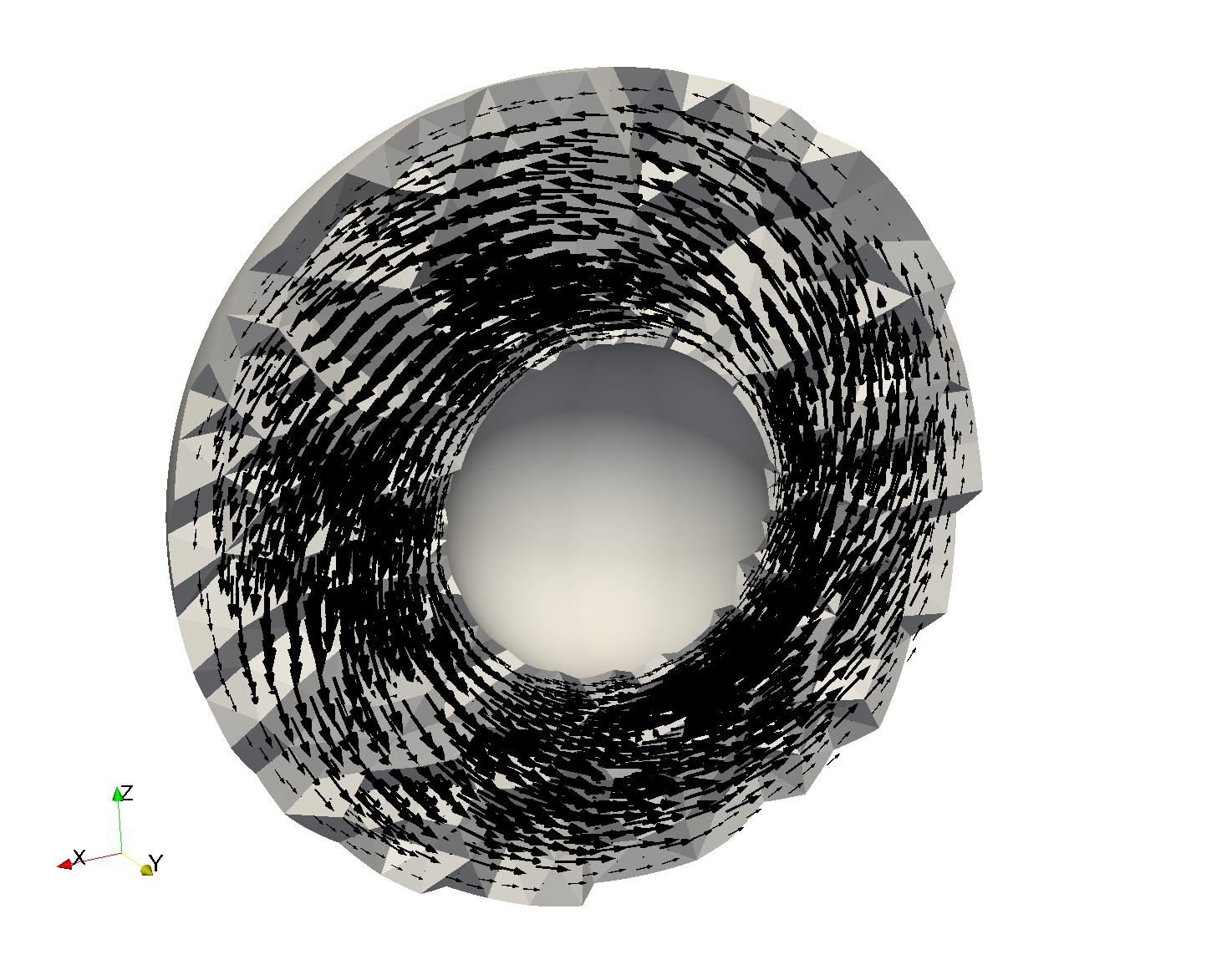}
  \includegraphics[width=0.48\textwidth]{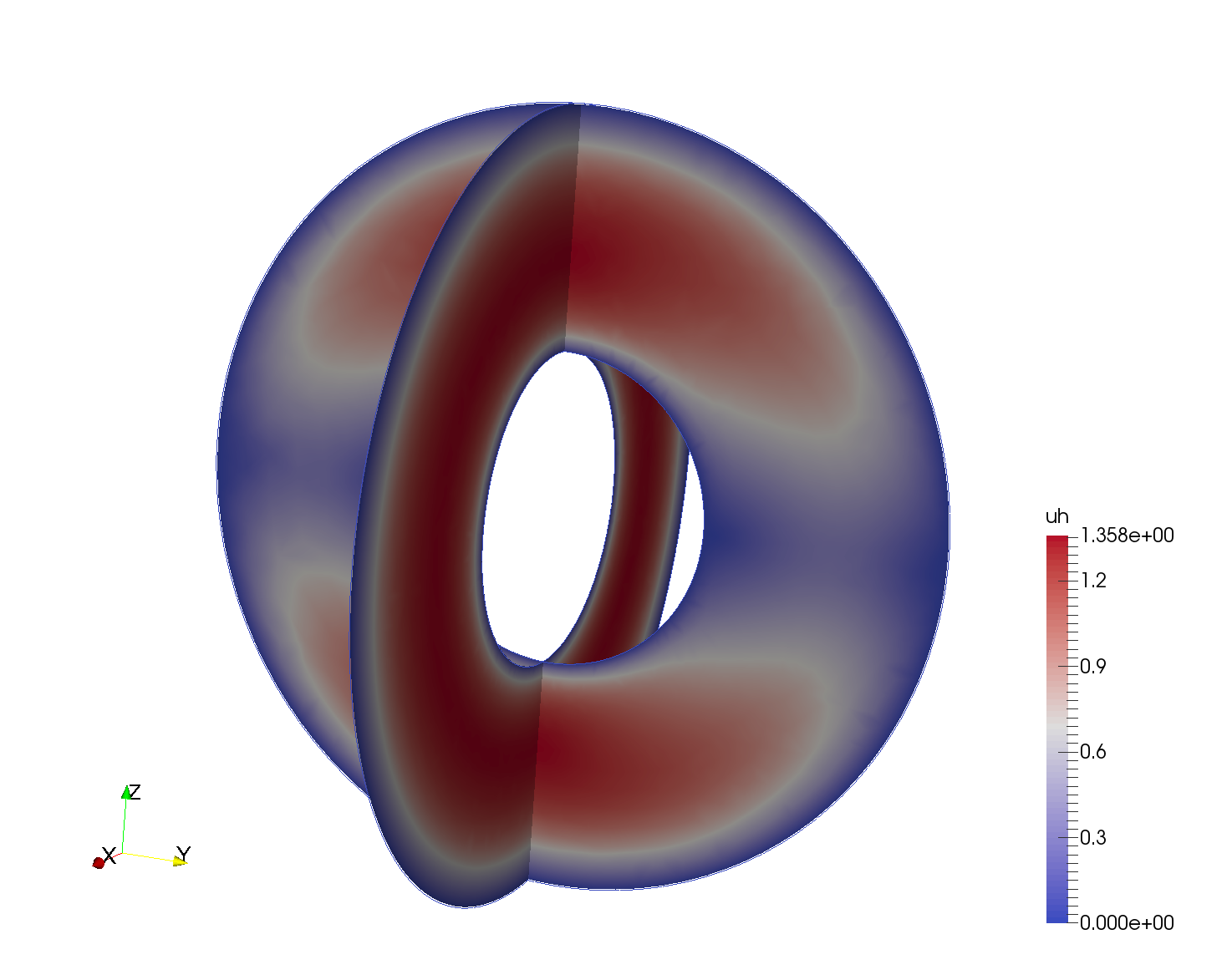}
  \caption{Velocity vector field in the vicinity of plane $y=0$ (left), velocity magnitude on the plane $x=0$ and $y=0$ (right).}\label{fig:velocity_dis}
\end{figure}
\end{example}
}
\section{Conclusion}
\label{sec:con}
This paper develops a parametric mixed finite element method with optimal convergence rates for the stationary incompressible inductionless MHD equations on 3D curved domains. The proposed method is charge-conservative in the sense that the discrete divergence-free constraint $\Div \BJ_h = 0$ is exactly preserved over the computational domain. We establish the well-posedness of both continuous and discrete variational formulations and prove the discrete inf-sup condition for the proposed scheme. Through a combination of geometric approximation error analysis and operator estimates, we derive optimal error bounds in the energy norm. Moreover, by using a Stokes projection, we also obtain an optimal $L^2$-norm error estimate for the velocity. Numerical experiments validate the theoretical results and demonstrate the accuracy and effectiveness of the proposed method. Several limitations of the present work should be noted. The analysis is restricted to stationary inductionless MHD and smooth curved domains, and uniqueness relies on small-data assumptions. Natural directions for future research include extensions to the full MHD system and time-dependent problems, as well as the treatment of non-smooth geometries and a refined study of the parameter dependence in the error estimates.

\vspace{5mm}

\section*{Appendix}
Since the discrete problem \eqref{eq:model_8st-0} contains a trilinear convection term, it constitutes a nonlinear system. We solve it by means of a Picard fixed point iteration, which linearizes the convective term while preserving the full coupling among all unknowns. The iteration is initialized with the zero guess.
Let $m\ge 1$ denote the iteration index. Given the previous iterates
\[(\Bu_h^{m-1}, p_h^{m-1}, \BJ_h^{m-1}, \phi_h^{m-1})\]
we replace $\Co(\Bu_h;\Bu_h,\Bv_h)$ with $\Co(\Bu_h^{m-1}; \Bu_h^m, \Bv_h)$. The linearized variational problem then reads: find $(\Bu_h^m, p_h^m, \BJ_h^m, \phi_h^m) \in \BV_{h,0}^k \times Q_h^{k-1} \times \BD_h^{k-1} \times S_h^{k-2}$ such that for all test functions $(\Bv_h, q_h, \Bd_h, \varphi_h) \in \BV_{h,0}^k \times Q_h^{k-1} \times \BD_h^{k-1} \times S_h^{k-2}$,
\begin{subequations}
\begin{align}
\Co(\Bu_h^{m-1}; \Bu_h^m, \Bv_h)_h + R_e^{-1}(\nabla\Bu_h^m, \nabla\Bv_h)_h + \gamma(\Div\Bu_h^m, \Div\Bv_h)_h \nn\\
 - (p_h^m, \Div\Bv_h)_h - \alpha (\BJ_h^m\times\BB, \Bv_h)_h = (\Bf_h, \Bv_h)_h, \label{eq:linear_mom} \\
(\BJ_h^m, \Bd_h)_h - (\phi_h^m, \Div\Bd_h)_h - (\Bu_h^m \times\BB, \Bd_h)_h = (\Bg_h, \Bd_h)_h, \label{eq:linear_ohm} \\
-(\Div\Bu_h^m, q_h)_h = 0, \label{eq:linear_cont} \\
-(\Div\BJ_h^m, \varphi_h)_h = 0. \label{eq:linear_charge}
\end{align}
\end{subequations}

The corresponding algebraic system takes the block form
\begin{equation*}
\mathbb{A}^{m} \xi^{m} = b^{m},
\end{equation*}
where $\xi^{m} = (\Bu_h^m, p_h^m, \BJ_h^m, \phi_h^m)^\top$ and
\begin{equation*}
\mathbb{A}^{m} =
\begin{pmatrix}
\BL_u + \BN_u^{m} & -(B_q)^\top & -\alpha\,\BC_J & 0\\
-B_q & 0 & 0 & 0\\
-\BC_u & 0 & \BQ_J & -(B_\phi)^\top\\
0 & 0 & -B_\phi & 0
\end{pmatrix}.
\end{equation*}
Here $\BL_u$ consists of the viscous and grad-div terms, $\BN_u^{m}$ is the convection operator, $\BC_u$ and $\BC_J$ are the Lorentz coupling matrices, $B_q$ and $B_\phi$ are the divergence constraint matrices for velocity and current density, and $\BQ_J$ is the mass matrix for $\BJ$. The right-hand side $b^{m}$ contains the source terms. The monolithic linear system is solved via a preconditioned Krylov subspace method with a relative residual $10^{-6}$.
The Picard iteration is terminated when the norm of the increments over all solution components is below a tolerance $10^{-8}$,
\begin{align*}
\|\Bu_h^m - \Bu_h^{m-1}\|_{\BL^2(\DOmega_h)} < 10^{-8},
\quad
\|p_h^m - p_h^{m-1}\|_{L^2(\DOmega_h)} < 10^{-8},\\
\|\BJ_h^m - \BJ_h^{m-1}\|_{\BL^2(\DOmega_h)} < 10^{-8},
\quad
\|\phi_h^m - \phi_h^{m-1}\|_{L^2(\DOmega_h)} < 10^{-8}.
\end{align*}
For simplicity, details of the preconditioner are omitted here and can be found in \cite{lz17,lnz19x}.

\input{refs}

\end{document}

%% file: macros.tex
\providecommand{\Div}{\operatorname{div}}          

\providecommand*{\Dist}[2]{\operatorname{dist}({#1};{#2})}   
\providecommand*{\Dist}[2]{\Dist{#1}{#2}}

\newcommand{\Bd}{{\boldsymbol{d}}}
\newcommand{\Be}{{\boldsymbol{e}}}
\newcommand{\Bf}{{\boldsymbol{f}}}
\newcommand{\Bg}{{\boldsymbol{g}}}

\newcommand{\Bj}{{\boldsymbol{j}}}

\newcommand{\Bn}{{\boldsymbol{n}}}

\newcommand{\Bu}{{\boldsymbol{u}}}
\newcommand{\Bv}{{\boldsymbol{v}}}
\newcommand{\Bw}{{\boldsymbol{w}}}
\newcommand{\Bx}{{\boldsymbol{x}}}

\newcommand{\hBd}{{\boldsymbol{\hat d}}}

\newcommand{\hBu}{{\boldsymbol{\hat u}}}
\newcommand{\hBv}{{\boldsymbol{\hat v}}}

\newcommand{\hBx}{{\boldsymbol{\hat x}}}

\newcommand{\tBf}{{\boldsymbol{\tilde f}}}
\newcommand{\tBg}{{\boldsymbol{\tilde g}}}

\newcommand{\tBu}{{\boldsymbol{\tilde u}}}
\newcommand{\tBv}{{\boldsymbol{\tilde v}}}

\newcommand{\hBJ}{{\boldsymbol{\hat J}}}

\newcommand{\tBJ}{{\boldsymbol{\tilde J}}}

\newcommand{\BB}{{\boldsymbol{B}}}
\newcommand{\BC}{{\boldsymbol{C}}}
\newcommand{\BD}{{\boldsymbol{D}}}

\newcommand{\BH}{{\boldsymbol{H}}}

\newcommand{\BJ}{{\boldsymbol{J}}}

\newcommand{\BL}{{\boldsymbol{L}}}
\newcommand{\BM}{{\boldsymbol{M}}}
\newcommand{\BN}{{\boldsymbol{N}}}

\newcommand{\BP}{{\boldsymbol{P}}}
\newcommand{\BQ}{{\boldsymbol{Q}}}

\newcommand{\BV}{{\boldsymbol{V}}}
\newcommand{\BW}{{\boldsymbol{W}}}

\newcommand{\Ca}{\mathcal{A}}
\newcommand{\Cb}{\mathcal{B}}

\newcommand{\Cd}{\mathcal{D}}

\newcommand{\Cf}{\mathcal{F}}

\newcommand{\Cl}{\mathcal{L}}

\newcommand{\Co}{\mathcal{O}}
\newcommand{\Cp}{\mathcal{P}}

\newcommand{\Cr}{\mathcal{R}}

\newcommand{\Ct}{\mathcal{T}}

\newcommand{\bbI}{\mathbb{I}}
\newcommand{\bbJ}{\mathbb{J}}

\newcommand{\bbN}{\mathbb{N}}

\providecommand*{\wh}[1]{\widehat{#1}}

\newcommand{\D}{\mathrm{d}}

\newcommand{\be}{\begin{eqnarray}}
\newcommand{\ee}{\end{eqnarray}}
\newcommand{\nn}{\nonumber}
\newcommand{\ben}{\begin{eqnarray*}}
\newcommand{\een}{\end{eqnarray*}}


%% file: refs.tex
\bibliographystyle{amsplain}

%% file: arXiv.bbl
\begin{thebibliography}{10}

\bibitem{rc23}
Aylwin, R., Jerez-Hanckes, C.:
Finite-element domain approximation for Maxwell variational problems on curved domains.
SIAM J. Numer. Anal. \textbf{61}(3), 1139-1171 (2023).
https://doi.org/10.1137/21M1468772

\bibitem{fg16}
Bertrand, F., Starke, G.:
Parametric Raviart-Thomas elements for mixed methods on domains with curved surfaces.
SIAM J. Numer. Anal. \textbf{54}(6), 3648-3667 (2016).
https://doi.org/10.1137/15M1045442


\bibitem{fsg14}
Bertrand, F., Munzenmaier, S., Starke, G.:
First-order system least squares on curved boundaries: High-order Raviart-Thomas elemetns.
SIAM J. Numer. Anal. \textbf{52}, 3165-3180 (2014).


\bibitem{MFEMbook}
Boffi, D., Brezzi, F., Fortin, M.:
Mixed finite element methods and applications.
Springer, Heidelberg (2013).
https://doi.org/10.1007/978-3-642-36519-5

\bibitem{TEbook}
Brenner, S.C., Scott, L.R.:
The mathematical theory of finite element methods, vol.~3.
Springer, New York (2008).
https://doi.org/10.1007/978-0-387-75934-0

\bibitem{clsw10}
Chen, G., Li, D., Sch\"{o}tzau, D., Wei, X.:
A mixed finite element method with exactly divergence-free velocities for incompressible magnetohydrodynamics.
Comput. Methods Appl. Mech. Engrg. \textbf{199}, 2840-2855 (2010).
https://doi.org/10.1016/j.cma.2010.05.007

\bibitem{cpg02}
Ciarlet, P.G.:
The finite element method for elliptic problems.
SIAM, Philadelphia (2002).
https://doi.org/10.1115/1.3424474

\bibitem{ch06}
Codina, R., Hern\'{a}ndez-Silva, N.:
Stabilized Finite Element Approximation of the Stationary Magneto-Hydrodynamics Equations.
Comput. Mech. \textbf{38}, 344-355 (2006).
https://doi.org/10.1007/s00466-006-0037-x

\bibitem{pa01}
Davidson, P.A.:
An introduction to magnetohydrodynamics.
Cambridge University Press, Cambridge (2001).
https://doi.org/10.1017/CBO9780511626333

\bibitem{fokd19}
Farrokhi, H., Otuya, D., Khimchenko, A., Dong, J.:
Magnetohydrodynamics in biomedical applications.
IntechOpen (2019).
https://doi.org/10.5772/intechopen.87109

\bibitem{gll06}
Gerbeau, J.F., LeBris, C., Leli\'{e}vre, T.:
Mathematical methods for the magnetohydrodynamics of liquid metals.
Oxford University Press, Oxford (2006).
https://doi.org/10.1093/acprof:oso/9780198566656.001.0001

\bibitem{NSbook}
Girault, V., Raviart, P.A.:
Finite element methods for Navier-Stokes equations.
Springer, New York (1986).
https://doi.org/10.1007/978-3-642-61623-5

\bibitem{jp19}
Goedbloed, H., Keppens, R., Poedts, S.:
Magnetohydrodynamics of laboratory and astrophysical plasmas.
Cambridge University Press, Cambridge (2019).
https://doi.org/10.1017/9781316403679

\bibitem{gmp91}
Gunzburger, M. D., Meir, A. J., Peterson, J. S.:
On the Existence, Uniqueness, and Finite Element Approximation of Solutions of the Equations of Stationary, Incompressible Magnetohydrodynamics.
Math. Comput. \textbf{56}, 523-563 (1991).
https://doi.org/10.1090/S0025-5718-1991-1066834-0

\bibitem{hlmz18}
Hiptmair, R., Li, L., Mao, S., Zheng, W.:
A fully divergence-free finite element method for magnetohydrodynamic equations.
Math. Models Methods Appl. Sci. \textbf{28}(4), 659-695 (2018).
https://doi.org/10.1142/S0218202518500173

\bibitem{jlmnr17}
John, V., Linke, A., Merdon, C., Neilan, M., Rebholz, L.G.:
On the divergence constraint in mixed finite element methods for incompressible flows.
SIAM Rev. \textbf{59}, 492-544 (2017).
https://doi.org/10.1137/15M1047696

\bibitem{ml86}
Lenoir, M.:
Optimal isoparametric finite elements and error estimates for domains involving curved boundaries.
SIAM J. Numer. Anal. \textbf{23}, 562-580 (1986).
https://doi.org/10.1137/0723036

\bibitem{lm89}
Lielpeteris, J., Moreau, R.:
Liquid metal magnetohydrodynamics.
Springer Dordrecht (1989).
https://doi.org/10.1007/978-94-009-0999-1

\bibitem{lz17}
Li, L., Zheng, W.:
A robust solver for the finite element approximation of stationary incompressible MHD equations in 3D,
J. Comput. Phys. \textbf{351}(15), 254-270 (2017).
https://doi.org/10.1016/j.jcp.2017.09.025

\bibitem{lnz19}
Li, L., Ni, M., Zheng, W.:
A charge-conservative Finite Element Method for inductionless MHD equations. Part I: Convergence.
SIAM J. Sci. Comput. \textbf{41}(4), B796-B815 (2019).
https://doi.org/10.1137/17M1160768

\bibitem{lnz19x}
Li, L., Ni, M., Zheng, W.:
A charge-conservative Finite Element Method for inductionless MHD equations. Part II: A robust solver.
SIAM J. Sci. Comput. \textbf{41}(4), B816--B842 (2019).
https://doi.org/10.1137/19M1260372

\bibitem{rm90}
Moreau, R.:
Magnetohydrodynamics.
Kluwer Academic Publishers, Dordrecht (1990).
https://doi.org/10.1007/978-94-015-7883-7

\bibitem{mn22}
Muir, H., Nikiforakis, N.:
Numerical modeling of imposed magnetohydrodynamic effects in hypersonic flows.
Phys. Fluids \textbf{34}(10), 107114 (2022).
https://doi.org/10.1063/5.0115424

\bibitem{nmh1}
Ni, M.J., Munipalli, R., Morley, N.B., Huang, P., Abdou, M.A.:
A current density conservative scheme for incompressible MHD flows at a low magnetic Reynolds number.
I. On a rectangular collocated grid system.
J. Comput. Phys. \textbf{227}(1), 174-204 (2007).
https://doi.org/10.1016/j.jcp.2007.07.025

\bibitem{nmh2}
Ni, M.J., Munipalli, R., Morley, N.B., Huang, P., Abdou, M.A.:
A current density conservative scheme for incompressible MHD flows at a low magnetic Reynolds number.
II. On an arbitrary collocated mesh.
J. Comput. Phys. \textbf{227}(1), 205-228 (2007).
https://doi.org/10.1016/j.jcp.2007.07.023

\bibitem{or04}
Olshanskii, M. A., Reusken, A.:
Grad-div stabilization for stokes equations.
Math. Comput. \textbf{73}, 1699-1718 (2004).

\bibitem{pbhn20}
Pamela, S.J.P., Bhole, A., Huijsmans, G.T.A., Nkonga, B., Hoelzl, M., Krebs, I., Strumberger, E.:
Extended full-MHD simulation of nonlinear instabilities in tokamak plasmas.
Phys. Plasmas \textbf{27}(10), 102510 (2020).
https://doi.org/10.1063/5.0018208

\bibitem{pjs88}
Peterson, J.S.:
On the finite element approximation of incompressible flows of an electrically conducting fluid.
Numer. Methods Partial Differ. Equ. \textbf{4}(1), 57-68 (1988).
https://doi.org/10.1002/num.1690040105

\bibitem{rp11}
Planas, R., Badia, S., Codina, R.:
Approximation of the inductionless MHD problem using a stabilized finite element method.
J. Comput. Phys. \textbf{230}(8), 2977-2996 (2011).
https://doi.org/10.1016/j.jcp.2010.12.046

\bibitem{rv21}
Ruas, V.:
Optimal-rate finite-element solution of Dirichlet problems in curved domains with straight-edged tetrahedra.
IMA J. of Numer. Anal. \textbf{41}(2), 1368-1410 (2021).
https://doi.org/10.1093/imanum/draa029

\bibitem{ds04}
Sch\"{o}tzau, D.:
Mixed finite element methods for stationary incompressible magnetohydrodynamics.
Numer. Math. \textbf{96}, 771-800 (2004).
https://doi.org/10.1007/s00211-003-0487-4

\bibitem{ss13}
Smolentsev, S., Vetcha, N., Abdou, M.:
Effect of a magnetic field on stability and transitions in liquid breeder flows in a blanket.
Fusion Eng. Des. \textbf{88}, 607-610 (2013).
https://doi.org/10.1016/j.fusengdes.2013.04.001

\bibitem{ZhangPHG}
Zhang, L.:
A parallel algorithm for adaptive local refinement of tetrahedral meshes using bisection.
Numer. Math. Theor. Meth. Appl. \textbf{2}, 65-89 (2009).
https://lsec.cc.ac.cn/phg/index.htm
\end{thebibliography}
